\documentclass[a4,12pt]{amsart}

\usepackage{amssymb}
\usepackage{amstext}
\usepackage{amsmath}
\usepackage{amscd}
\usepackage{amsthm}
\usepackage{amsfonts}
\usepackage{enumerate}
\usepackage{graphicx}
\usepackage{latexsym}
\usepackage[all]{xy}
\usepackage[usenames]{color}
\usepackage{mathrsfs}
\usepackage{color}

\newenvironment{red}
{\relax\color{red}}
{\hspace*{.5ex}\relax}

\newcommand{\ber}{\begin{red}}
\newcommand{\er}{\end{red}}

\newenvironment{verd}
{\relax\color{magenta}}
{\hspace*{.5ex}\relax}

\newcommand{\bg}{\begin{verd}}
\newcommand{\eg}{\end{verd}}

\newtheorem*{corollary*}{Corollary}
\newtheorem{theorem}{Theorem}[section]

\newtheorem{lemma}[theorem]{Lemma}
\newtheorem{proposition}[theorem]{Proposition}
\newtheorem*{proposition*}{Proposition}

\newtheorem*{claim*}{Claim}

\theoremstyle{definition}
\newtheorem{definition}[theorem]{Definition}
\newtheorem{remark}[theorem]{Remark}
\newtheorem{example}[theorem]{Example}

\newtheorem{definitiontheorem}[theorem]{Definition-Theorem}

\theoremstyle{remark}

\numberwithin{equation}{theorem}
\renewcommand{\mod}{\operatorname{mod}}

\newcommand{\End}{\operatorname{End}}
\newcommand{\Hom}{\operatorname{Hom}}

\newcommand{\Ext}{\operatorname{Ext}}

\newcommand{\K}{\mathcal{K}}

\renewcommand{\O}{\mathcal{O}}

\newcommand{\Z}{{\mathbb{Z}}}
\newcommand{\Q}{{\mathbb{Q}}}

\newcommand{\Ker}{\operatorname{Ker}}
\renewcommand{\Im}{\operatorname{Im}}
\newcommand{\Coker}{\operatorname{Coker}}

\newcommand{\Rad}{\operatorname{Rad}}

\newcommand{\Soc}{\operatorname{Soc}}

\newcommand{\Endhom}{\operatorname{End}}

\newcommand{\Der}{\operatorname{Der}}

\begin{document}
\title[On components of stable AR quivers that contain Heller lattices]{On components of stable Auslander-Reiten quivers that contain Heller lattices: the case of truncated polynomial rings}
\date{}

\author[Susumu Ariki]{Susumu Ariki}
\address{Department of Pure and Applied Mathematics, Graduate School of Information
Science and Technology, Osaka University,  1-5 Yamadaoka, Suita, Osaka 565-0871, Japan}
\email{ariki@ist.osaka-u.ac.jp}

\author[Ryoichi Kase]{Ryoichi Kase}
\address{Faculty of Informatics, Okayama University of Science, 1-1 Ridaicho, Kita-ku, Okayama-shi 700-0005, Japan }
\email{r-kase@mis.ous.ac.jp}

\author[Kengo Miyamoto]{Kengo Miyamoto}
\address{Department of Pure and Applied Mathematics, Graduate School of Information
Science and Technology, Osaka University, 1-5 Yamadaoka, Suita, Osaka 565-0871, Japan}
\email{k-miyamoto@ist.osaka-u.ac.jp}

\thanks{
Before we started this project, the first author had asked his student Takuya Takeuchi for some experimental computation for $n=3$ case. 
We thank him for this computation at the preliminary stage of the research. 
}

%\urladdr{}
%\thanks{2010 {\em Mathematics Subject Classification.} Primary ; Secondary }

\begin{abstract}
Let $A$ be a truncated polynomial ring over a complete discrete valuation ring $\O$, and we consider the additive category consisting of 
$A$-lattices $M$ with the property that $M\otimes \K$ is projective as an $A\otimes \K$-module, where $\K$ is the fraction field of $\O$. 
Then, we may define the stable Auslander-Reiten quiver of the category. 
We determine the shape of the components of the stable Auslander-Reiten quiver that contain Heller lattices. 
\end{abstract}
\maketitle
%\tableofcontents

%%%%%%%%%%%%%%%%%%%%%%%%%%
\section*{Introduction} %%
%%%%%%%%%%%%%%%%%%%%%%%%%%

The shape of Auslander-Reiten quivers is one of fundamental interests in representation theory of algebras. 
For algebras over a field, a wealth of examples are given in textbooks, \cite{ASS} for example. Let $\O$ be a complete discrete valuation ring, 
$\epsilon$ a uniformizer, 
$\K$ its fraction field, $\kappa=\O/\epsilon\O$ its residue field. Let $A$ be an $\O$-order, namely an $\O$-algebra which is free of finite rank 
as an $\O$-module. If $A\otimes \K$ is a semisimple algebra, we may also find results in the literature. 
However, few results seem to be known for the case when  $A\otimes \K$ is not a semisimple algebra. An exception is a famous work by Hijikata and Nishida, but 
their main focus is on a Bass order and $A\otimes \K$ needs to be a quasi-Frobenius radical square zero algebra for a Bass order \cite[Thm.3.7.1]{HN}.

Recall that an $A$-module is called an $A$-lattice or a Cohen-Macaulay $A$-module if it is free of finite rank as an $\O$-module. 
(Cohen-Macaulay $A$-modules are by definition finitely generated $A$-modules which are Cohen-Macaulay as $\O$-modules. 
Since $\O$ is regular here, Cohen-Macaulay $\O$-modules are free \cite[(1.5)]{Y} and vice versa.) 
Then, it is known that for any non-projective $A$-lattice $M$ with the property that $M\otimes \K$ is projective as an $A\otimes \K$-module, there is an 
almost split sequence ending at $M$, and dually, for any non-injective $A$-lattice $M$ with the property that $M\otimes \K$ is injective as an $A\otimes \K$-module, there is an 
almost split sequence starting at $M$. See \cite{AR} for example. Thus, if $A\otimes \K$ is self-injective, we may define the (stable) Auslander-Reiten quiver consisting of 
such $A$-lattices. 
Typical examples of such $A$-lattices are Heller lattices. For group algebras, Heller lattices were studied by Kawata \cite{K}, and 
it inspired us to study the components that contain Heller lattices for the case of orders in non-semisimple algebras. 

In this article, we determine the shape of the components of the stable 
Auslander-Reiten quiver that contain Heller lattices, for the truncated polynomial rings $A=\O[X]/(X^n)$. 
As $\O[X]/(X^n)$ is a Gorenstein $\O$-order, that is, $\Hom_\O(A_A, \O)$ is a projective $A$-module \cite[\S4]{I}, we explain explicit construction of 
almost split sequences for a Gorenstein $\O$-order, which generalizes construction of 
almost split sequences in \cite{T}, and use this construction to do necessary calculations. Main difficulty in the computation is the proof that 
certain direct summands of the middle terms of those almost split sequences are indecomposable. We use elementary 
brute force argument to overcome this difficulty. Then, some argument on tree classes which takes the possibility of the existence of loops in the stable 
Auslander-Reiten quiver into account proves the result. This argument is necessary because there may exist loops \cite{W}. 

If $A\otimes \kappa$ is a special biserial algebra, we may calculate indecomposable $A\otimes \kappa$-modules and their Heller lattices. It is natural to 
consider the above problem in this setting. We will report some results in this direction in future work.

%%%%%%%%%%%%%%%%%%%%%%%%%%
\section{Preliminaries} %%
%%%%%%%%%%%%%%%%%%%%%%%%%%

\subsection{Gorenstein orders}

We start by observing that $A=\O[X]/(X^n)$ is a symmetric $\O$-order. By abuse of notation, we write $1,X, \dots, X^{n-1}$ for the standard $\O$-basis of $A$. 
Define $\theta_{i}\in \Hom_\O(A,\O)$, for $0\le i\le n-1$, by 
\[
\theta_{i}(X^{j})=\left\{\begin{array}{cl}
1 & \mathrm{if\ }j=n-i-1,\\
0 & \mathrm{if\ }j\neq n-i-1.\\
\end{array}\right.
\]
Then we have the following lemma. 
\begin{lemma}\label{symmetric order}
$\theta_{i}\mapsto X^{i}$ induces an isomorphism of $(A,A)$-bimodules $\Hom_\O(A,\O)\simeq A$.
\end{lemma}
\begin{proof}
As $X\theta_{i}=\theta_{i}X: X^{j}\mapsto \theta_{i}(X^{j+1})=\delta_{j+1,n-i-1}$, we have $X\theta_{i}=\theta_{i}X=\theta_{i+1}$. 
\end{proof}

\begin{remark}
A different definition of Gorenstein order is given in \cite[\S37]{CR}: it requires not only that every exact sequence of $A$-lattices  
$0 \to A \to M \to N \to 0$ starting at $A$ splits, but also that $A\otimes\K$ is a semisimple algebra. Perhaps the semisimplicity condition was added by 
some technical reasons.
\end{remark}

\begin{remark}
In \cite[Chap. I, \S7]{A1}, the definition of $\O$-order itself is different. If we restrict to the case when $\O$ is a Dedekind domain, $A$ is an $\O$-order 
in his sense if $A$ is not only a finitely generated projective $\O$-module but also $A\otimes\K$ is a self-injective $\K$-algebra. 

Then, a Gorenstein $\O$-order is a Noetherian $\O$-algebra $A$ which is Cohen-Macaulay as an $\O$-module 
and $\Hom_\O(A,\O)\simeq A$ as $(A,A)$-bimodules \cite[Chap. III, \S1]{A1}. Nowadays, Gorenstein $\O$-orders in Auslander's sense are called 
symmetric $\O$-orders \cite[Def.2.8]{IW}.

Lemma \ref{symmetric order} implies that $A=\O[X]/(X^n)$ is a symmetric $\O$-order. Note that 
$A$ is also a Gorenstein ring, since ${\rm depth}\,A=\dim A$ and if the parameter ideal $\epsilon A$ is the intersection of two ideals $I$ and $J$
then either $I=\epsilon A$ or $J=\epsilon A$ holds.
\end{remark}

\begin{lemma}\label{infinite lattices}
Let $A=\O[X]/(X^n)$, for $n\geq 2$. Then there are infinitely many pairwise non-isomorphic indecomposable $A$-lattices. 
\end{lemma}
\begin{proof}
If there were only finitely many, then \cite[\S10]{A2} and \cite[(3.1), (4.22)]{Y} would imply that $A$ is reduced, contradicting our assumption that $n\geq 2$. Below, we give an example of a family of infinitely many pairwise non-isomorphic indecomposable $A$-lattices.

For $r\in \Z_{\geq 0}$, let $L_{r}=\O\epsilon^{r}\oplus \O X\oplus\cdots\oplus \O X^{n-1}\subseteq A$. Then 
the representing matrix of the action of $X$ on $L_{r}$ with respect to the basis is given by the following matrix:
\[
X = \begin{pmatrix}
 0 & \cdots &\cdots  &\cdots  & 0\\
\epsilon^{r} & 0 & \cdots &\cdots & 0 \\
 0 & 1 & 0      &  &   \vdots         \\
\vdots & \ddots & \ddots &\ddots & \vdots \\
 0 & \cdots & 0 & 1&0 \end{pmatrix}
\]
Therefore we have $L_{r}\otimes\K\simeq A\otimes \K$ and $L_{r}\not\simeq L_{s}$ whenever $r\neq s$.
 In particular, $L_{r}$, for $r=0,1,2,\cdots$, are pairwise non-isomorphic indecomposable $A$-lattices.
\end{proof}

Since $\O$ is a complete local ring, $\End _A(X)$ is a local $\O$-algebra for every indecomposable $A$-lattice $X$ \cite[(6.10)(30.5)]{CR}. 
Thus, the Jacobson radical $\Rad\End _A(X)$ consists of all non-invertible endomorphisms of $X$. 
Another consequence is that $A$ is semiperfect and every finitely generated $A$-module has a projective cover \cite[(6.23)]{CR}. 

In the next subsection, we assume that $A$ is a Gorenstein $\O$-order and we explain a method to construct almost split sequences for 
$A$-lattices. Note that there exists an almost split sequence ending (resp. starting) at $M$ if and only if $M\otimes\K$ is projective (resp. injective) 
\cite{AR}, \cite[Thm.6]{RR}. 

\subsection{Construction of almost split sequences}

We recall several definitions. 

\begin{definition} Let $A$ be an $\O$-order, $M$ and $N$ $A$-lattices. 
The \emph{radical} $\Rad \Hom _A(M,N)$ of $\Hom _A(M,N)$ is the $\O$-submodule of $\Hom _A(M,N)$ consisting of 
$f\in \Hom _A(M,N)$ such that, for all indecomposable $A$-lattice $X$, we have
$hfg\in\Rad\End _A(X)$, for any $g\in \Hom_A(X, M)$ and $h\in \Hom_A(N, X)$. 
It is equivalent to the condition that $1-gf$ is invertible, for all $g\in\Hom_A(N,M)$, and 
to the condition that $1-fg$ is invertible, for all $g\in\Hom_A(N,M)$.
\end{definition}

Let $\mathcal{A}$ be an abelian category with enough projectives, $\mathcal{C}$ an additive full subcategory which is closed under 
extensions and direct summands. Then, $f \in \Hom_{\mathcal C}(M,N)$ in $\mathcal{C}$ is called \emph{right minimal in $\mathcal{C}$} if 
an endomorphism $h\in \End_{\mathcal C}(M)$ is an isomorphism whenever $f=fh$, \emph{right almost split in $\mathcal{C}$} 
if it is not a split epimorphism and for each $X\in\mathcal{C}$ and $h\in \Hom_{\mathcal C}(X,N)$ which is not a split epimorphism, there is $s\in \Hom_{\mathcal C}(X,M)$ 
such that $fs=h$. If $f$ is both right minimal in $\mathcal{C}$ and right almost split in $\mathcal{C}$, $f$ is called 
\emph{minimal right almost split in $\mathcal{C}$}. Similarly, $g \in \Hom_{\mathcal C}(L,M)$ is called \emph{left minimal in $\mathcal{C}$} if 
an endomorphism $h\in \End_{\mathcal C}(M)$ is an isomorphism whenever $g=hg$, \emph{left almost split in $\mathcal{C}$}  
if it is not a split monomorphism and for each $Y\in\mathcal{C}$ and $h\in \Hom_{\mathcal C}(L,Y)$ which is not a split monomorphism, there is $t\in \Hom_{\mathcal C}(M,Y)$ 
such that $tg=h$, and if $g$ is both left minimal in $\mathcal{C}$ and left almost split in $\mathcal{C}$, $g$ is called 
\emph{minimal left almost split in $\mathcal{C}$}. We have the following proposition in this general setting \cite[Chap.II, Prop.4.4]{A1}. 

\begin{proposition}\label{equivalent conditions}
Suppose that $\mathcal{C}$ is an additive full subcategory of an abelian category $\mathcal{A}$ with enough projectives such that 
$\mathcal{C}$ is closed under extensions and direct summands. 
Let $L,M,N\in\mathcal{C}$. Then the following are equivalent for a short exact sequence
$$ 0 \longrightarrow L \stackrel{g}{\longrightarrow} M \stackrel{f}{\longrightarrow} N \longrightarrow 0. $$
\begin{itemize}
\item[(a)]
$f$ is right almost split in $\mathcal{C}$ and $g$ is left almost split in $\mathcal{C}$.
\item[(b)]
$f$ is minimal right almost split in $\mathcal{C}$.
\item[(c)]
$f$ is right almost split and $\End_{\mathcal C}(L)$ is local.
\item[(d)]
$g$ is minimal left almost split in $\mathcal{C}$.
\item[(e)]
$g$ is left almost split in $\mathcal{C}$ and $\End_{\mathcal C}(N)$ is local.
\end{itemize}
\end{proposition}

We return to $\O$-orders over a complete discrete valuation ring $\O$. Among equivalent conditions in Proposition \ref{equivalent conditions}, 
we choose (c) as the definition of an almost split sequence for lattices over an $\O$-order. 

\begin{definition}
Let $A$ be an $\O$-order, $L,E,M$ $A$-lattices. A short exact sequence
\[ 0 \longrightarrow L \longrightarrow E \xrightarrow{\ p\ } M \longrightarrow 0 \]
is called an \emph{almost split sequence} (of $A$-lattices) ending at $M$  if 
\begin{itemize}
\item[(i)] the epimorphism $p$ does not split,
\item[(ii)] $L$ and $M$ are indecomposable, 
\item[(iii)] the morphism $p: E\to M$ induces the epimorphism 
\[ \Hom _A(X,p):\Hom _A(X,E) \longrightarrow \Rad\Hom _A(X,M), \]
for every indecomposable $A$-lattice $X$.
\end{itemize}
\end{definition}

\begin{definition}
Let $f:M\to N$ be a morphism between $A$-lattices. We say that $f$ is an \emph{irreducible morphism} if 
\begin{itemize}
\item[(i)]
$f$ is neither a split monomorphism nor a split epimorphism,
\item[(ii)]
if there are $g\in \Hom_A(M,L)$ and $h\in \Hom_A(L,N)$ such that $f=hg$, then either $g$ is a split monomorphism or $h$ is a split 
epimorphism.
\end{itemize}
\end{definition}

\begin{lemma}
Let $A$ be an $\O$-order, $L,E,M$ $A$-lattices. We suppose that an almost split sequence for $A$-lattices ending at $M$ exists. Then, 
a short exact sequence
\[ 0 \longrightarrow L \xrightarrow{\ \iota\ } E \xrightarrow{\ p\ } M \longrightarrow 0 \]
is an almost split sequence if and only if $\iota$ and $p$ are irreducible.
\end{lemma}
\begin{proof}
The arguments in \cite[Thm. 5.3]{ARS} and \cite[Prop. 5.9]{ARS} work without change in our setting.
\end{proof}

\begin{remark}
The definitions of almost split sequences and irreducible morphisms are taken from \cite{R2}, although it is assumed that $A\otimes\K$ is a semisimple algebra there. 
\end{remark}

\begin{definition}
Let $A$ be an $\O$-order. For an indecomposable $A\otimes\kappa$-module $N$, 
we view $N$ as an $A$-module, and take the projective cover $p: P\to N$. We denote $\Ker(p)$ by $Z_N$ and 
direct summands of the $A$-lattice $Z_N$ are called \emph{Heller lattices} of $N$. 
Note that $Z_N$ is uniquely determined up to isomorphism. 
\end{definition}

In the sequel, we consider an indecomposable $A$-lattice $M$ with the property
\begin{quote}\label{asterisk}
($\ast$)\quad $M\otimes \K$ is projective as an $A\otimes \K$-module,
\end{quote}
and show how to construct the almost split sequence ending at $M$. 

\begin{remark}
Heller lattices have the property $(\ast)$. Indeed, for an indecomposable $A\otimes\kappa$-module $N$, $Z_N$ is an $A$-submodule of the projective $A$-module 
$P$, and we have $\epsilon P\subseteq Z_N$. Thus, $Z_N\otimes \K=P\otimes \K$ is projective and so are their direct summands.
\end{remark}

Let $D=\Hom_\O(-,\O)$ and define the \emph{Nakayama functor} for $A$-lattices by 
$$ \nu = D(\Hom_A(-,A))=\Hom_{\O}(\Hom_A(-,A),\O). $$

\begin{lemma}\label{first exact sequence}
Let $M$ be an $A$-lattice, $p:P\to M$ its projective cover. We define 
$$ L=D(\Coker(\Hom _A(p,A))). $$
Then we have the exact sequence of $A$-lattices
$$
0\longrightarrow L \longrightarrow \nu (P) \stackrel{\nu(p)}{\longrightarrow} \nu (M) \longrightarrow 0.
$$
\end{lemma}
\begin{proof}
$\Hom _A(\Ker(p),A)$ is an $A$-lattice since $\Ker(p)$ and $A$ are. 
Since the cokernel of $\Hom _A(p,A):\Hom_A(M,A) \to \Hom_A(P,A)$ is an $A$-submodule of $\Hom _A(\Ker(p),A)$, 
$\Coker(\Hom _A(p,A))$ is a free $\O$-module. 
Then, $\Ext_{\O}^1(\Coker(\Hom _A(p,A)),\O)=0$ implies the result. 
\end{proof}

\begin{remark}
If we take a minimal projective presentation $Q\stackrel{q}{\rightarrow}P\stackrel{p}{\rightarrow} M$ of an $A$-lattice $M$, we have the short exact sequence
$$
0 \rightarrow \Coker(\Hom _A(p,A)) \rightarrow \Hom_A(Q,A) \rightarrow \Coker(\Hom _A(q,A))=\operatorname{Tr}(M) \rightarrow 0.
$$
Thus, $L=D(\Coker(\Hom _A(p,A)))$ represents the Auslander-Reiten translate $\tau(M)=D\Omega\operatorname{Tr}(M)$ of the $A$-lattice $M$. 
\end{remark}

Taking a suitable pullback of the exact sequence from Lemma \ref{first exact sequence}, we may construct almost split 
sequences as follows. This generalizes the construction in \cite{T}. We give the proof of Proposition \ref{construction of AR-seq}
in the appendix, for the convenience of the reader. 

The right and left minimality in Proposition \ref{equivalent conditions} implies that 
the almost split sequence ending at $M$ and the almost split sequence starting at $L$ are uniquely determined by $M$ and $L$ respectively, 
up to isomorphism of short exact sequences. Thus, We may define the Auslander-Reiten translate $\tau$ and $\tau^-$ by $\tau(M)=L$ 
and $\tau^-(L)=M$. 

\begin{proposition}\label{construction of AR-seq}
Suppose that $A$ is a Gorenstein $\O$-order, $M$ an indecomosable non-projective $A$-lattice with the property $(\ast)$, and 
let $p: P\rightarrow M$ be its projective cover. 
For $\varphi\in \Hom_A(M, \nu (M))$, we consider the pullback diagram along $\varphi$:

\bigskip
\hspace{20mm}
\begin{xy}
(0,15)*[o]+{0}="01",(20,15)*[o]+{L}="L",(40,15)*[o]+{E}="E", (60,15)*[o]+{M}="M",(80,15)*[o]+{0}="02",
(0,0)*[o]+{0}="03",(20,0)*[o]+{L}="L2",(40,0)*[o]+{\nu(P)}="nP", (60,0)*[o]+{\nu(M)}="nM",(80,0)*[o]+{0}="04",
\ar "01";"L"
\ar "L";"E"
\ar "E";"M"
\ar "M";"02"
\ar "03";"L2"
\ar "L2";"nP"
\ar "nP";"nM"^{\nu(p)}
\ar "nM";"04"
\ar @{-}@<0.5mm>"L";"L2"
\ar @{-}@<-0.5mm>"L";"L2"
\ar "E";"nP"
\ar "M";"nM"^{\varphi}
\end{xy}

\bigskip
\noindent
Then the following (1) and (2) are equivalent.
\begin{itemize}
\item[(1)] The pullback $0\to L \to E \to M \to 0$ is an almost split sequence.
\item[(2)] The following three conditions hold.
\begin{itemize}
\item[(i)] $\varphi$ does not factor through $\nu(p)$. 
\item[(ii)] $L$ is an indecomposable $A$-lattice.
\item[(iii)] For all $f\in \Rad\End _A(M)$, $\varphi f$ factors through $\nu(p)$.
\end{itemize} 
\end{itemize}
\end{proposition}

If $A$ is a symmetric $\O$-order, then we have functorial isomorphisms $\nu(X)\simeq X$, for $A$-lattices $X$. 
Hence, we pull back $0\to L\to P \to M\to 0$ along $\varphi\in \End_A(M)$ in this case. Further, the left term $L=\tau(M)$ and the middle term $E$ of 
the almost split sequence satisfy the property $(\ast)$.

\subsection{Translation quivers and tree classes}
In this subsection we recall fundamentals of translation quivers.

\begin{definition}
Let $Q=(Q_0, Q_1)$, where $Q_{0}$ is the set of vertices and $Q_{1}$ is the set of arrows, be a locally finite quiver, that is, 
there are only finitely many incoming and outgoing arrows for each vertex. 
If a map $v:Q_{1}\rightarrow \Z_{\geq 0}\times \Z_{\geq 0}$ is given, 
we call the pair $(Q,v)$ a \emph{valued quiver}. Let $\tau:Q\rightarrow Q$ be a quiver automorphism. Then, 
we call the pair $(Q,\tau)$ a \emph{stable translation quiver} if the following two conditions hold:
\begin{itemize}
\item[(i)] 
$Q$ has no loops and no multiple arrows.
\item[(ii)] 
For each vertex $x\in Q_{0}$, we have
$$ \{y\in Q_{0}\mid \tau x\rightarrow y \text{ in } Q_{1}\}=\{y\in Q_{0}\mid y\rightarrow x \text{ in } Q_{1}\}.$$
\end{itemize}
We call the triple $(Q,v,\tau)$ a \emph{valued stable translation quiver} if $(Q,\tau)$ is a stable translation quiver and if
$v(x\to y)=(a,b)$ then $v(\tau(y)\to x)=(b,a)$.
\end{definition}

\begin{definition}
Let $(Q,\tau)$ be a stable translation quiver and $C$ a full subquiver of $Q$. We call $C$ a \emph{component} of $(Q,\tau)$ if 
\begin{itemize}
\item[(i)]
$C$ is stable under the quiver automorphism $\tau$. 
\item[(ii)]
$C$ is a disjoint union of connected components of the underlying undirected graph. 
\item[(iii)]
There is no proper subquiver of $C$ that satisfies $(i)$ and $(ii)$. 
\end{itemize}
Note that components are also stable translation quivers. 
\end{definition}

\begin{example}
Let $(\Delta,v)$ be a valued quiver without loops and multiple arrows. Then, the set $\Z\times\Delta$ becomes a valued stable translation quiver by defining 
as follows. 
\begin{itemize}
\item
arrows are $(n,x)\to (n,y)$ and $(n-1,y)\to (n,x)$, for $x\to y$ in $\Delta$ and $n\in \Z$.
\item
if $v(x\to y)=(a,b)$, for $x\to y$ in $\Delta$, then 
$$ v((n,x)\to (n,y))=(a,b)\;\; \text{ and }\;\; v((n-1,y)\to (n,x))=(b,a). $$
\item
$\tau((n,x))=(n-1,x)$.
\end{itemize}
We denote the valued stable translation quiver by $\Z\Delta$.
\end{example}

Now we recall Riedtmann's structure theorem \cite[Thm.4.15.6]{B}. For the definition of admissible subgroups, see \cite[Def.4.15.4]{B}. 

\begin{definitiontheorem}
Let $(Q,\tau)$ be a stable translation quiver and $C$ a component of $(Q,\tau)$. Then there is a directed tree $T$ and an admissible 
subgroup $G\subseteq \mathrm{Aut}(\Z T)$ such that $C\simeq \Z T/G$ as a stable translation quiver. Moreover, 
\begin{itemize}
\item[(1)] the underlying undirected graph $\overline{T}$ of $T$ is uniquely determined by $C$.
\item[(2)] $G$ is unique up to conjugation in $\mathrm{Aut}(\Z T)$.
\end{itemize}
The underlying tree $\overline{T}$ is called the \emph{tree class} of $C$.  
\end{definitiontheorem}

\begin{definition}
Let $(\Delta,v)$ be a valued quiver without loops and multiple arrows. 
For $x\to y$ in $\Delta$, we write $v(x\to y)=(d_{xy}, d'_{xy})$. If there is no arrow between $x$ and $y$, we understand that 
$d_{xy}=d'_{xy}=0$. Let $\Q_{>0}$ be the set of positive rational numbers. Let
\[ x^+=\{ y\in \Delta_0\mid x\to y \in \Delta_1\}, \quad x^-=\{y\in \Delta_0\mid y\to x\in \Delta_1\}. \]
\begin{itemize}
\item[(i)]
A \emph{subadditive function} on $(\Delta,v)$ is a $\Q_{>0}$-valued function $f$ on $\Delta_0$ such that 
\[ 2f(x)\ge \sum_{y\in x^-} d_{yx}f(y)+\sum_{y\in x^+} d'_{xy}f(y), \]
for each vertex $x\in \Delta _0$.
\item[(ii)]
A subadditive function is an \emph{additive function} if the equality holds for all $x\in \Delta _0$.
\end{itemize}
\end{definition}

The following lemma is well known. See \cite[Thm.4.5.8]{B}, for example. 

\begin{lemma}
\label{tree class}
Let $(\Delta,v)$ be a valued quiver without loops and multiple arrows, and we assume that the underlying undirected graph 
$\overline{\Delta}$ is connected. 
\begin{itemize}
\item[(1)]
Suppose that $(\Delta,v)$ admits a subadditive function. 
\begin{itemize}
\item[(i)]
If $\Delta$ has a finite number of vertices, then $\overline{\Delta}$ is one of finite or affine Dynkin diagrams.
\item[(ii)]
If $\Delta$ has infinite number of vertices, then $\overline{\Delta}$ is one of infinite Dynkin diagrams
$A_{\infty},\; B_{\infty},\; C_{\infty},\; D_{\infty}$ or $A_{\infty}^{\infty}$. 
\end{itemize}
\item[(2)]
If $(\Delta,v)$ admits a subadditive function which is not additive, then $\overline{\Delta}$ is either a finite Dynkin 
diagram or $A_{\infty}$.
\item[(3)]
$(\Delta,v)$ admits a subadditive function which is unbounded, then $\overline{\Delta}$ is $A_{\infty}$.
\end{itemize}
\end{lemma}

\subsection{AR-quivers}

We define the stable Auslander-Reiten quiver for symmetric $\O$-orders as follows.

\begin{definition}
Let $A$ be a symmetric $\O$-order over a complete discrete valuation ring $\O$. The \emph{stable Auslander-Reiten quiver} of $A$ is a valued quiver 
such that 
\begin{itemize}
\item
vertices are isoclasses of non-projective $A$-lattices $M$ such that $M\otimes \K$ is projective. 
\item
valued arrows $M\stackrel{(a,b)}{\rightarrow} N$ for irreducible morphisms $M\to N$, 
where the value $(a,b)$ of the arrow is given as follows. 
\begin{itemize}
\item[(a)]
For a minimal right almost split morphism $f: E\to N$, $M$ appears $a$ times in $E$ as a direct summand. 
\item[(b)]
For a minimal left almost split morphism $g: M\to E$, $N$ appears $b$ times in $E$ as a direct summand.
\end{itemize}
\end{itemize}
A component of the stable Auslander-Reiten quiver is defined in the similar way as the stable translation quiver. 
\end{definition}

\begin{lemma}\label{no loop lemma}
Let $A$ be a symmetric $\O$-order over a complete discrete valuation ring $\O$, and let 
$C$ be a component of the stable Auslander-Reiten quiver of $A$. Assume that $C$ satisfies the following conditions:
\begin{enumerate}[(i)]
\item There exists a $\tau$-periodic indecomposable $A$-lattice in $C$.
\item $C$ has infinitely many vertices.
\end{enumerate}
Then $C\setminus\{\text{loops}\}$ is of the form $\mathbb{Z}A_{\infty}/\langle \tau \rangle$ if $C$ has a loop. In this case, the deleted loops appear only at the endpoint of $C$.
\end{lemma}
\begin{proof}
First, we show that if $X\in C$ has a loop, then $X\simeq \tau X$. Suppose that $X\in C$ has a loop and $X\not\simeq \tau X$. Then  
the almost split sequence ending at $X$ is of the form
\[ 0 \to \tau X\to X ^{\oplus l_1} \oplus E_X \oplus  \tau X^{\oplus l_2}\to X\to 0, \]
where $E_X$ is an $A$-lattice and $l_1, l_2\geq 1$. Then, we have
\[ {\rm rank}(X)+{\rm rank}( \tau X)=l_1{\rm rank}( X)+l_2{\rm rank}( \tau X)+{\rm rank}( E_X), \]
hence ${\rm rank}( E_X)=0$ and $l_1=l_2=1$. However, it follows from \cite[Theorem 1]{M} that the almost split sequence ending at $X$ splits, a contradiction.
Therefore, if $X$ has a loop, then $X$ and $\tau X$ are isomorphic.

As in the proof of \cite[Thm.4.16.2]{B}, we know that all indecomposable $A$-lattices in $C$ are $\tau$-periodic. 
Thus, we may choose $n_X\geq 1$, for each $X\in C$, such that $\tau^{n_X}(X)\simeq X$. Define a $\Q_{>0}$-valued function $f$ on $C$ by
$$
f(X)= \frac{1}{n_X}\sum_{i=0}^{n_X-1}{\rm rank}\;\tau^{i}(X).
$$

$C$ does not have multiple arrows by definition. For each indecomposable $N$, 
there is an irreducible morphism $M\to N$ if and only if there is an irreducible morphism $\tau(N)\to M$ by 
the existence of the almost split sequence $0 \to \tau(N) \to E \to N \to 0$. The condition on valued arrows may also be 
checked. Thus, $C\setminus\{\rm{loops}\}$ is a valued stable translation quiver, and we may apply 
the Riedtmann structure theorem. We write $C\setminus\{{\rm loops}\}=\Z T/G$, for a directed tree $T$ and an admissible subgroup $G$. 
Then $f$ is a $\Q_{>0}$-valued function on $T$. For $X\in T$, one can show that
\[ 2f(X)\geq \sum_{Y\in X^-\cap T} d_{YX}f(Y)+\sum_{Y\in X^+\cap T} d'_{XY}f(Y), \]
which implies that $f$ is a subadditive function on $T$. 

We now suppose that $C$ has a loop. Then, $f$ is not additive. Thus, Lemma \ref{tree class} and our assumption (ii) imply that 
$\overline{T}=A_\infty$. Thus, we may assume without loss of generality that $T$ is a chain of irreducible maps
$$
X_{1}\rightarrow X_{2}\rightarrow\cdots \rightarrow X_{r}\rightarrow \cdots.
$$
We assume that $X_r$ has a loop. If $r>1$ then the almost split sequence starting at $X_r$ is
\[ 0 \longrightarrow X_r \longrightarrow X_r^{\oplus l} \oplus X_{r+1}\oplus X_{r-1} \oplus P \longrightarrow X_r \longrightarrow 0 \]
where $l \geq 1$ and $P$ is a projective $A$-module.  Since $f(X_t)\geq 1$  for all $t\geq 1$, we have
\[ f(X_r) \geq  (2- l)f(X_r) \geq f(X_{r+1}) + f(X_{r-1}) \geq f(X_{r+1})+1. \] 
We show that $f(X_m) \geq  f(X_{m+1})+1$ for $m \geq r$. Suppose that $f(X_{m-1}) \geq
f(X_m)+1$ holds. The same argument as above shows $2f(X_m) \geq f(X_{m-1}) +
f(X_{m+1})$, and the induction hypothesis implies 
\[ 2f(X_m)\geq f(X_{m-1}) + f(X_{m+1}) \geq f(X_m) + f(X_{m+1})+1. \]
Hence $f(X_m) \geq f(X_{m+1})+1$.
Thus, there exists a positive integer $t$ such that $f(X_t)<0$, which contradicts with $f(X_t)\geq 1$. Hence $r=1$, that is, the deleted loops appear only at the endpoint of the homogeneous tube.
\end{proof}

\subsection{No loop theorem}

In this subsection, we show an analogue of Auslander's theorem and use this to show \lq\lq no loop theorem\rq\rq. 

\begin{lemma}
Let $A$ be an $\O$-order, $M$ an indecomposable $A$-lattice. Then, there exists an integer $s$ such that 
$M/\epsilon^k M$ is an indecomposable $A/\epsilon^k A$-module, for all $k\geq s$. 
\end{lemma}
\begin{proof}
An  $\O$-linear map $D:A \rightarrow \End _{\O}(M)$ is called a derivation if 
\[ D(xy)=xD(y)+D(x)y \]
for all $x,y\in A$. We denote by $\Der(A,\End _{\O}(M))$ the $\O$-module of derivations. 
Note that $\Der(A,\End _{\O}(M))$ is an $\O$-order since $A$ and $M$ are.

Let $k$ be a positive integer. For $f\in \End_{\O}(M)$ such that $af(m+\epsilon^k M)=f(am+\epsilon^k M)$, for $a\in A$ and $m\in M$, 
we define $D_f\in \Hom_\O(A,\End_{\O}(M))$ as follows. 
\[ D_{f}(a)(m) = \epsilon ^{-k}(f(am)-af(m)),\;\;\text{for $a\in A$ and $m\in M$.} \]
The following computation shows that $D_{f}$ is a derivation. 
\begin{align*}
D_{f}(xy)(m)&=\epsilon ^{-k}(f(xym)-xy(m)) \\
                     &= \epsilon ^{-k}(xf(ym)-xyf(m))+\epsilon ^{-k}(f(xym)-xf(ym)) \\
                     &= xD_{f}(y)(m)+D_{f}(x)(ym). \end{align*}
Let $\Der(k)$ be the $\O$-submodule of  $\Der(A,\End _{\O}(M))$ which is generated by all such $D_f$, and we define 
$\Der(\infty)=\sum_{k\geq 1}\Der(k)$. Since $\Der(A,\End _{\O}(M))$ is a finitely generated $\O$-module, there exists an integer $s$ such that
\[ \Der(\infty)=\sum _{k=1} ^{s-1}\Der(k). \]
We show that the algebra homomorphism $\End _A(M) \rightarrow \End _A(M/\epsilon ^kM)$ is surjective, for all $k\geq s$. 
Let $\theta \in \End _{A}(M/\epsilon ^kM)$, for $k\geq s$. We fix $f\in \End_\O(M)$ such that 
$$ f(m+\epsilon^kM)=\theta(m+\epsilon^kM), \;\;\text{for $m\in M$}. $$
Then, there exist $c_i\in \O$ and $f_i\in\End_\O(M)$ that satisfy 
$$ f_i(m+\epsilon^{l_i}M)=\theta_i(m+\epsilon^{l_i}M), \;\;\text{for some $1\leq l_i\leq s-1$ and $\theta_i\in \End_A(M/\epsilon^{l_i}M)$},$$
such that $D_{f}=\sum _{i=1}^{N}c_iD_{f_i}$. More explicitly, we have
\[ f(am)-af(m)=\sum _{i=1} ^{N} \epsilon ^{k-l_i}c_i(f_{i}(am)-af_{i}(m)), \;\;\text{for $a\in A$ and $m\in M$.}\]
It implies that $f-\sum _{i=1}^{N}\epsilon ^{k-l_i}c_if_{i}\in\End _A(M)$. Since it coincides with $\theta$ if we reduce modulo $\epsilon$,
we have proved
$$ \Im(\End_A(M)\to \End _A(M/\epsilon ^kM))+\epsilon\End _A(M/\epsilon ^kM)=\End _A(M/\epsilon ^kM). $$
Thus, Nakayama's lemma implies that $\End_A(M)\to \End _A(M/\epsilon ^kM)$ is surjective, and we have 
an isomorphism of algebras $\End_A(M)/\epsilon^k\End_A(M)\simeq\End_A(M/\epsilon^k M)$. 
As $\O$ is a complete local ring, the lifting idempotent argument works \cite[(6.7)]{CR}. Hence,  
if $M/\epsilon ^k M$ is decomposable, so is $M$.
\end{proof}

We recall the Harada-Sai lemma from \cite[VI. Cor.1.3]{ARS}.

\begin{lemma}\label{Harada-Sai}
Let $B$ be an Artin algebra, $\{ N_i \mid 1\leq i\leq 2^m\}$ a collection of indecomposable $B$-modules such that 
the length of composition series of $N_i$ is less than or equal to $m$, for all $i$. 
If none of $f_i\in \Hom_B(N_i,N_{i+1})$ $(1\leq i\leq 2^m-1)$ is an isomorphism, then
$$ f_{2^m-1}\cdots f_1=0. $$
\end{lemma}

\begin{proposition}\label{lattice Gabriel}
Let $A$ be a symmetric $\O$-order over a complete discrete valuation ring $\O$, and assume that $A$ is indecomposable as an $\O$-algebra. 
Let $C$ be a component of the stable Auslander-Reiten quiver of $A$. 
Assume that the number of vertices in $C$ is finite. Then $C$ exhausts all non-projective indecomposable $A$-lattices. 
\end{proposition}
\begin{proof}
We add indecomposable projective $A$-lattices to the stable Auslander-Reiten quiver of $A$ to obtain the Auslander-Reiten quiver of $A$. 
We show that if $C$ is a finite component of the Auslander-Reiten quiver then $C$ exhausts all indecomposable $A$-lattices. 
Assume that $M$ is an indecomposable $A$-lattice which does not belong to $C$. It suffices to show 
\[ \Hom _A(M,N) =0 = \Hom _A(N,M), \;\; \text{for all $N\in C$}. \]
To see that it is sufficient, let $P$ be a direct summand of the projective cover of $N\in C$.  
Then, $P\in C$ by $N \in C$ and $\Hom _A(P,N)\neq 0$. As $A$ is indecomposable as an algebra, there is no indecomposable projective $A$-lattice $Q$
with the property that 
$$ \Hom _A(Q,R) =0 = \Hom _A(R,Q), $$
for all indecomposable projective $A$-lattices $R\in C$. It implies that any direct summand $Q$ 
of the projective cover of $M$ belongs to $C$. Then $\Hom_A(Q,M)\neq 0$ implies that $M\in C$, which contradicts our assumption. Thus, 
$C$ exhausts all indecomposable $A$-lattices.

Assume that there exists a nonzero morphism $f\in \Hom _A(M,N)$. As $M\not\in C$ and $N\in C$, $f$ is not a split epimorphism. 
We consider the almost split sequence of $A$-lattices 
ending at $N$, and we denote by $N_1,\dots, N_r$ the indecomposable direct summands of the middle term of the almost split sequence. Let 
$$ g_i ^{(1)}:N_i \longrightarrow N $$
be irreducible morphisms. Then, there exist $f_i\in \Hom_A(M, N_i)$ such that 
$$ f=\sum_{i=1}^r g_i ^{(1)} f_i. $$
If $N_i$ is non-projective, we apply the same procedure to $f_i$. If $N_i$ is projective, $f_i$ factors through the Heller lattice $\Rad N_i$ of the 
irreducible $A\otimes\kappa$-module $N_i/\Rad(N_i)$. Thus, we apply the procedure after we replace $N_i$ with $\Rad N_i$. 
After repeating $n$ times, we obtain,  
\[ f=\sum g_i ^{(1)}\cdots g_i^{(n)} h_i, \]
such that $g_i^{(j)}$ are morphisms among indecomposable $A$-lattices in $C$, $h_i$ are morphisms $M \to X_i$, 
where $X_i$ are indecomposable $A$-lattices in $C$ and they are not isomorphisms. 

Since the number of vertices in $C$ is finite, there exists an integer $s$ such that $X/\epsilon ^sX$ is indecomposable, for all $X\in C$. 
Let $m$ be the maximal length of $A/\epsilon^s A$-modules $X/\epsilon ^sX$, for $X\in C$. 
Applying Lemma \ref{Harada-Sai} to the Artin algebra $A/\epsilon ^sA$ with $n=2^m-1$, we obtain
\[ \Hom _A(M,N)=\epsilon ^s\Hom _A(M,N), \]
and Nalayama's Lemma implies $\Hom _A(M,N)=0$. The proof of $\Hom _A(N,M)=0$ is similar. We start with a nonzero morphism $f\in \Hom _A(N,M)$ and consider 
the almost split sequence of $A$-lattices starting at $N$. Let $N_1,\dots, N_r$ be the indecomposable direct summands of the middle term of the almost split sequence as above, and let
$$ g_i ^{(1)}:N \longrightarrow N_i $$
be irreducible morphisms. If $N_i$ is projective, then we replace $N_i$ with $\Rad N_i$. Then, after repeating the procedure $n$ times, we obtain
\[ f=\sum h_i g_i ^{(n)}\cdots g_i^{(1)}, \]
where $h_i$ are morphisms from indecomposable $A$-lattices in $C$ to $M$. Then, we may deduce  $\Hom _A(N,M)=0$ by the Harada-Sai lemma and Nakayama's lemma as before.
\end{proof}

\begin{theorem}\label{infinite tree class theorem}
Let $A$ be a symmetric $\mathcal{O}$-order over a complete discrete valuation ring $\mathcal{O}$, and let $C$ be a component of the stable Auslander-Reiten quiver of $A$. Suppose that
\begin{enumerate}[(i)]
\item there exists a $\tau$-periodic indecomposable $A$-lattice in $C$,
\item the stable Auslander-Reiten quiver of $A$ has infinitely many vertices.
\end{enumerate}
Then, the number of vertices of $C$ is infinite, and either 
\begin{enumerate}[(a)]
\item $C$ is a valued stable translation quiver, or
\item $C\setminus\{\text{loops}\}=\mathbb{Z}A_{\infty}/\langle \tau\rangle$ and the deleted loops appear only at the endpoint of $C$.
\end{enumerate}
\end{theorem}
\begin{proof}
As in the proof of Lemma \ref{no loop lemma},  $C$ admits a subadditive function by the condition (i). 
Hence, the tree class of the valued stable translation quiver $C\setminus\{{\rm loops}\}$ is one of finite, affine or infinite Dynkin diagrams. 
In the first two cases, the number of vertices in $C$ is finite, since all vertices in $C$ are $\tau$-periodic. 
Then we may apply Proposition \ref{lattice Gabriel} and it contradicts the condition (ii). Thus, the tree class is one of infinite Dynkin diagrams and 
the number of vertices in $C$ is infinite. Now the assertion follows from Lemma \ref{no loop lemma}.
\end{proof}

%%%%%%%%%%%%%%%%%%%%%%%%%%%%%%%%%%%%%
\section{The case $A=\O[X]/(X^n)$} %%
%%%%%%%%%%%%%%%%%%%%%%%%%%%%%%%%%%%%%

\subsection{Heller lattices}

Let $M_{i}=\kappa[X]/(X^{n-i})$, for $1\le i\le n-1$. They form a complete set of isoclasses of non-projective indecomposable 
$A\otimes\kappa$-modules. We realize $M_{i}$ as the $A\otimes\kappa$-submodule $X^iA+\epsilon A/\epsilon A$ of $A\otimes\kappa=A/\epsilon A$. 
We view $M_{i}$ as an $A$-module. 
Then, $p:A\twoheadrightarrow M_{i}$ defined by $f\mapsto X^{i}f+\epsilon A$ is the projective cover of $M_{i}$. 
Therefore, the Heller lattice $Z_{i}$ of $M_{i}$, which is an $A$-submodule of $A$, is given as follows:
$$
Z_{i}=\O\epsilon\oplus\O\epsilon X\oplus\cdots\O\epsilon X^{n-i-1}\oplus\O X^{n-i}\oplus\O X^{n-i+1}\oplus\cdots\oplus \O X^{n-1}.
$$
Then the representing matrix of the action of $X$ on $Z_{i}$ with respect to the above basis is given by the following matrix:
\[
X = \bordermatrix{
 &   &        &       & {n-i} & & & \cr
 & 0 & \cdots &\cdots  &\vdots  &\cdots  &\cdots &\cdots  & 0\cr
 & 1 & \ddots &        &\vdots  &        &       &        & \vdots\cr
 &   & \ddots & 0      &\vdots  &        &       &        &       \cr
 &   &  & 1 &0 & & & & \vdots\cr
{n-i+1}&\cdots & \cdots & \cdots & \epsilon&0 & & &\vdots\cr
 & & & & & 1&0 & & \cr
 & \vdots & & & & &\ddots &\ddots &\vdots\cr
 & 0 & \cdots & \cdots & & & & 1  & 0\cr
}
\vspace*{6mm}
\]
Thus, $\Endhom_{A}(Z_{i})\simeq \{M\in \mathrm{Mat}(n,\O)\mid MX=XM\}$ is a local $\O$-algebra, since the right hand side is contained in 
\[
\left\{
\bordermatrix{
&  &  & & \cr
 & a & 0&\cdots & 0 \cr
 & & \ddots & \ddots&\vdots \cr
 &  &\ast  & \ddots &0 \cr
 & & & &a\cr
},
\quad a\in \O\right\}.
\]
It follows the next lemma. Note that $\rho\in \Endhom_{A}(Z_{i})$ is determined by $\rho(\epsilon)\in Z_{i}$. 

\begin{lemma}\label{z-indec}
We have the following.
\begin{itemize}
\item[(1)]
The Heller lattices $Z_{i}$ are pairwise non-isomorphic indecomposable $A$-lattices.
\item[(2)]
If $\rho\in \Rad\Endhom_{A}(Z_{i})$ then $\rho(\epsilon)$ has the form 
\[\rho(\epsilon)=a_{0}\epsilon+\cdots +a_{n-i-1}\epsilon X^{n-i-1}+a_{n-i}X^{n-i}+\cdots + a_{n-1} X^{n-1},\]
where $a_i\in\O$, for $1\le i\le n-1$, and $a_{0}\in \epsilon\O$.
\end{itemize}
\end{lemma}

\bigskip

We now consider the following pullback diagram:

\bigskip
\hspace{20mm}
\begin{xy}
(0,10) *{0}="A1", (15,10) *{Z_{n-i}}="B1", (30,10) *{E_i}="C1", (45,10) *{Z_i}="D1", (60,10) *{0}="E1",
(0,0) *{0}="A2", (15,0) *{Z_{n-i}}="B2", (30,0) *{A\oplus A}="C2", (45,0) *{Z_i}="D2", (60,0) *{0}="E2",

\ar "A1";"B1"
\ar "B1";"C1"
\ar "C1";"D1"
\ar "D1";"E1"

\ar "A2";"B2"
\ar "B2";"C2"^\iota
\ar "C2";"D2"^{\;\;\pi}
\ar "D2";"E2"

\ar @{=} "B1";"B2"
\ar "C1";"C2"
\ar "D1";"D2"^{\phi}
\end{xy}

\bigskip
\noindent
where $\phi$ is defined by $\phi(\epsilon)=X^{n-1}$ and 
$$
\phi(\epsilon X)=\cdots=\phi(\epsilon X^{n-i-1})=\phi(X^{n-i})=\cdots=\phi(X^{n-1})=0,
$$
$\pi(f,g)=X^{n-i}f-\epsilon g$, for $(f,g)\in A\oplus A$, and $\iota$ is given as follows. 
$$
\begin{array}{rlll}
\iota(\epsilon X^j)&=& (\epsilon X^j, X^{n-i+j}) &\text{ if } 0\le j\le i-1, \\
\iota(X^j)&=& (X^j,0) &\text{ if } i\le j\le n-1.
\end{array}
$$

\begin{remark}
Using the exact sequences
$$
0\to Z_{n-i}\to A\oplus A\to Z_i\to 0 \quad \text{and} \quad 0\to Z_{n-1}\to A\to \kappa\to 0,
$$
one computes 
$$
\Ext_A^i(\kappa, A)=\begin{cases} \kappa \quad &\text{if $i=1$},\\ 
                                              0         & \text{otherwise.}\end{cases}
$$
\end{remark}

\begin{lemma}
\label{z-ass} 
We have the following. 
\begin{itemize}
\item[(1)] $\phi$ does not factor through $\pi$.
\item[(2)] For any $\rho\in \Rad\Endhom_{A}(Z_{i})$, $\phi\rho$ factors through $\pi$.
\end{itemize}
\end{lemma}
\begin{proof}
(1)\;Suppose that there is a morphism $\mu=(\mu_{1},\mu_{2}):Z_{i}\rightarrow  A\oplus A$ such that 
$\pi\mu=\phi$. Then we have $X^{n-i}\mu_{1}(\epsilon)-\epsilon\mu_{2}(\epsilon)=\epsilon(\mu_{1}(X^{n-i})-\mu_{2}(\epsilon))=X^{n-1}$.
This is a contradiction.

(2)\;Write $\rho(\epsilon)=a_{0}\epsilon+\cdots + a_{n-i-1}\epsilon X^{n-i-1}+a_{n-i}X^{n-i}+\cdots + a_{n-1} X^{n-1}$. 
Then, by Lemma\;\ref{z-indec}, there exists $a\in \O$ such that $a_{0}=\epsilon a$. We define $\mu\in \Hom_A(Z_i, A\oplus A)$
by $\mu(\epsilon)=(0,-aX^{n-1})$. Then, it is easy to check that $\pi\mu=\phi\rho$ holds.
\end{proof}

By Proposition \ref{construction of AR-seq} and Lemma \ref{z-ass}, we have an almost split sequence
\[0\rightarrow Z_{n-i}\rightarrow E_{i}\rightarrow Z_{i}\rightarrow 0,\]
where $E_{i}=\{(f,g,h)\in A\oplus A\oplus Z_{i}\mid \pi(f,g)=\phi(h)\}$ is given by
\begin{gather*}
E_{i}=\O(\epsilon,X^{n-i},0)\oplus\O(\epsilon X,X^{n-i+1},0)\oplus\cdots \oplus\O(\epsilon X^{i-1},X^{n-1},0)\\
\quad\quad\oplus\; \O(X^{i},0,0)\oplus\O(X^{i+1},0,0)\oplus\cdots\oplus \O(X^{n-1},0,0)\\
\quad\qquad\quad\oplus\; \O(X^{i-1},0,\epsilon)\oplus\O(0,0,\epsilon X)\oplus\cdots\oplus\O(0,0,\epsilon X^{n-i-1})\\
\qquad\qquad\qquad\oplus\; \O(0,0,X^{n-i})\oplus\O(0,0,X^{n-i+1})\oplus\cdots \oplus\O(0,0,X^{n-1}).
\end{gather*}
To simplify the notation, we define $a_0=b_0=0$ and 
\begin{align*}
a_{k}&=\left\{\begin{array}{ll}
(X^{n-k},0,0)& \mathrm{if}\ 1\leq k\leq n-i,\\
(\epsilon X^{n-k},X^{2n-k-i},0)& \mathrm{if}\ n-i<k\leq n,\\
\end{array}\right.\\
b_{k}&=\left\{\begin{array}{ll}
(0,0,X^{n-k})& \mathrm{if}\ 1\leq k\leq i,\\
(0,0,\epsilon X^{n-k})& \mathrm{if}\ i<k< n,\\
(X^{i-1},0,\epsilon )& \mathrm{if}\ k=n.\\
\end{array}\right.
\end{align*}
Then, we have
\begin{align*}
Xa_k&=\left\{\begin{array}{ll}
a_{k-1}& (k\neq n-i+1)\\
\epsilon a_{k-1}& (k=n-i+1)\\
\end{array}\right.\\
Xb_k&=\left\{\begin{array}{ll}
b_{k-1}& (k\neq i+1,n)\\
\epsilon b_{k-1}& (k=i+1)\\
a_{n-i}+b_{n-1}& (k=n)\\
\end{array}\right.
\end{align*}
and
$$
\Ker(X^{k})=\bigoplus_{1\le j\le k}(\O a_{j}\oplus \O b_{j}).
$$

\subsection{Almost split sequence ending at $\mathbf{Z_{i}}$}
In this subsection, we show that the middle term $E_{i}$ of the almost split sequence
$$
0\rightarrow Z_{n-i}\rightarrow E_{i}\rightarrow Z_{i}\rightarrow 0
$$
is indecomposable, for $2\leq i\leq n-1$.

\begin{proposition}
\label{E-indec}
We have the following.
\begin{itemize}
\item[(1)] $A$ is an indecomposable direct summand of $E_{1}$.
\item[(2)] For $2\leq i\leq n-1$, $E_{i}$ are indecomposable $A$-lattices. 
\end{itemize}
\end{proposition}
\begin{proof} (1)\; As $Z_{n-1}=\Rad A$, it follows from \cite[Chap.III, Thm.2.5]{A1}. 
We also give more explicit computational proof here.
Define $x_{k},y_{k}\in E_{1}$, for $1\leq k\leq n$, as follows:
\begin{align*}
x_{k}&=\left\{\begin{array}{ll}
a_{1}+\epsilon b_{1}& \mathrm{if}\ k=1,\\
a_{k}+b_{k}& \mathrm{if}\ 2\leq k\leq n-1,\\
b_{n}& \mathrm{if}\ k=n,\\
\end{array}\right.\\
y_{k}&=\left\{\begin{array}{ll}
b_{k}& \mathrm{if}\ 1\leq k\leq n-1,\\
a_{n}-\epsilon b_{n}& \mathrm{if}\ k=n,\\
\end{array}\right.
\end{align*} 
Then they form an $\O$-basis of $E_{1}$. Moreover, we have $Xx_1=0$ and $Xy_1=0$,
\[Xx_{k}=x_{k-1},\;\text{for $2\leq k\leq n$}, \;\;\text{and}\;\; X y_{k}=\left\{\begin{array}{ll}
\epsilon y_{1}& \mathrm{if}\ k=2,\\
y_{k-1}& \mathrm{if}\ 3\leq k\leq n-1,\\
-\epsilon y_{n-1} & \mathrm{if}\ k=n.\\
\end{array}\right.
\]
Thus, the $\O$-span of $\{ x_k \mid 1\leq k\leq n\}$ is isomorphic to the indecomposable projective $A$-lattice $A$. 
In particular, $A$ is an indecomposable direct summand of $E_{1}$, and the other direct summand is indecomposable, because 
it becomes $A\otimes\K$ after tensoring with $\K$.

(2)\;$E_{n-1}$ does not have a projective direct summand by \cite[Chap.III, Thm.2.5]{A1}. Thus, 
\cite[Chap. III, Prop.1.7, Prop.1.8]{A1} and (1) imply that $E_{n-1}\simeq \tau(E_1)$ is indecomposable. 
We assume $2\leq i\leq n-2$ in the rest of the proof. 

Suppose that $E_i=E'\oplus E''$ with $E',E''\neq 0$. Since $Z_i\otimes\K=Z_{n-i}\otimes\K=A\otimes\K$, 
we have $E_i\otimes\K\simeq A\otimes\K\oplus A\otimes\K$, 
which implies that
$$
E'\otimes\K\simeq A\otimes\K\simeq E''\otimes\K. 
$$
In particular, rank\;$E'=n=$rank\;$E''$. Since 
$$
0 \to E'\cap \Ker(X^k) \to E' \to \Im(X^k) \to 0
$$
and $\Im(X^k)$ is a free $\O$-module, we have the increasing sequence of $\O$-submodules 
$$
0 \subsetneq \cdots \subsetneq E'\cap \Ker(X^k) \subsetneq E'\cap \Ker(X^{k+1}) \subsetneq \cdots \subsetneq E'\cap \Ker(X^n)=E'
$$
such that all the $\O$-submodules are direct summands of $E'$ as $\O$-modules. Thus, we may choose an $\O$-basis $\{e'_{k}\}_{1\leq k\leq n}$ 
such that $e'_{k}\in E'\cap\Ker(X^{k})\setminus \Ker(X^{k-1})$. Similarly, we may choose an $\O$-basis $\{e''_{k}\}_{1\leq k\leq n}$ of 
$E''$ such that $e''_{k}\in E''\cap\Ker(X^{k})\setminus \Ker(X^{k-1})$.
Write
$$
\begin{array}{lll}
e'_{k}&=&\alpha_{k}a_{k}+\beta_{k}b_{k}+A'_{k}, \;\text{ for $\alpha_k, \beta_k\in\O$ and $A'_{k}\in \Ker(X^{k-1})$,}\\
e''_{k}&=&\gamma_{k}a_{k}+\delta_{k}b_{k}+A''_{k}, \;\;\text{ for $\gamma_k, \delta_k\in\O$ and $A''_{k}\in \Ker(X^{k-1})$.}
\end{array}
$$
Without loss of generality, we may assume
$$
A'_{k}\in \Ker(X^{k-1})\cap E'',\quad A''_{k}\in \Ker(X^{k-1})\cap E'.
$$
Since $\{e'_{k}, e''_{k}\}$ and $\{a_k, b_k\}$ are $\O$-bases of $\Ker(X^k)/\Ker(X^{k-1})$, we have $\alpha_{k}\delta_{k}-\beta_{k}\gamma_{k}\not\in \epsilon \O$.

As $Xe'_{k}\in \Ker(X^{k-1})\cap E'$, there are $f^{(k)}_{k-1},\cdots, f^{(k)}_{1}\in \O$ such that
$$ Xe'_{k}=f^{(k)}_{k-1}e'_{k-1}+\cdots+f^{(k)}_{1}e'_{1}.$$
Similarly, there are $g^{(k)}_{k-1},\cdots, g^{(k)}_{1}\in \O$ such that
$$ Xe''_{k}=g^{(k)}_{k-1}e''_{k-1}+\cdots+g^{(k)}_{1}e''_{1}.$$
The coefficient of $a_{k-1}$ in $Xe'_{k}$ is given by
\[\left\{\begin{array}{cl}
\alpha_{k}& \mathrm{if}\ k\neq n-i+1,\\
\epsilon \alpha_{k} & \mathrm{if}\ k=n-i+1.\\
\end{array}\right.
\]
Thus, we have 
\[f^{(k)}_{k-1}\alpha_{k-1}=\left\{\begin{array}{cl}
\alpha_{k}& \mathrm{if}\ k\neq n-i+1,\\
\epsilon \alpha_{k} & \mathrm{if}\ k=n-i+1.\\
\end{array}\right.\]
Similarly, we have the following.

\begin{align*}
f^{(k)}_{k-1}\beta_{k-1}&=\left\{\begin{array}{cl}
\beta_{k}& \mathrm{if}\ k\neq i+1,\\
\epsilon \beta_{k} & \mathrm{if}\ k=i+1.\\
\end{array}\right.\\[5pt]
g^{(k)}_{k-1}\gamma_{k-1}&=\left\{\begin{array}{cl}
\gamma_{k}& \mathrm{if}\ k\neq n-i+1,\\
\epsilon \gamma_{k} & \mathrm{if}\ k=n-i+1.\\
\end{array}\right.\\[5pt]
g^{(k)}_{k-1}\delta_{k-1}&=\left\{\begin{array}{cl}
\delta_{k}& \mathrm{if}\ k\neq i+1,\\
\epsilon \delta_{k} & \mathrm{if}\ k=i+1.\\
\end{array}\right.\\ 
\end{align*}

We shall deduce a contradiction in the following three cases and conclude that 
$E_i$ is indecomposable, for $2\leq i\leq n-2$. 
\begin{quote}
\begin{itemize}
\item[(case a)]
$2\leq n-i<i$.
\item[(case b)]
$2\leq i=n-i$.
\item[(case c)]
$2\leq i<n-i$.
\end{itemize}
\end{quote}

Suppose that we are in (case a). 
We multiply each of $e'_{k}$ and $e''_{k}$ by suitable invertible elements to get new $\O$-bases of $E'$ and $E''$ in order to have the equalities 
$$
f^{(k)}_{k-1}=\left\{\begin{array}{cl}
1& \mathrm{if}\ k\neq n-i+1,\\
\epsilon & \mathrm{if}\ k=n-i+1,\\
\end{array}\right.
\;\; \mathrm{and}\quad g^{(k)}_{k-1}=\left\{\begin{array}{cl}
1& \mathrm{if}\ k\neq i+1,\\
\epsilon & \mathrm{if}\ k=i+1,\\
\end{array}\right.
$$
in the new bases. For $k=1$, we keep the original basis elements $e'_1$ and $e''_1$. 
Suppose that we have already chosen new $e'_{j}$ and $e''_{j}$, for $1\leq j\leq k-1$. 
If $k\neq n-i+1, i+1$, then 
$$
f^{(k)}_{k-1}g^{(k)}_{k-1}(\alpha_{k-1}\delta_{k-1}-\beta_{k-1}\gamma_{k-1})=\alpha_{k}\delta_{k}-\beta_{k}\gamma_{k}
$$
implies that $f^{(k)}_{k-1}$ and $g^{(k)}_{k-1}$ are invertible. Thus, multiplying $e'_{k}$ and $e''_{k}$ with 
their inverses respectively, we have $f^{(k)}_{k-1}=1$, $g^{(k)}_{k-1}=1$ in the new basis. Note that we have
$$
\begin{pmatrix} \alpha_1 & \beta_1 \\ \gamma_1 & \delta_1 \end{pmatrix}
=\begin{pmatrix} \alpha_2 & \beta_2 \\ \gamma_2 & \delta_2 \end{pmatrix}
=\cdots\cdot
=\begin{pmatrix} \alpha_{n-i} & \beta_{n-i} \\ \gamma_{n-i} & \delta_{n-i} \end{pmatrix}.
$$
If $k=n-i+1$, then, by using $i\neq n-i$, we have 
\begin{align*}
f^{(n-i+1)}_{n-i}g^{(n-i+1)}_{n-i}\alpha_{n-i}\delta_{n-i}&= \epsilon \alpha_{n-i+1}\delta_{n-i+1},\\
f^{(n-i+1)}_{n-i}g^{(n-i+1)}_{n-i}\beta_{n-i}\gamma_{n-i}&= \epsilon \beta_{n-i+1}\gamma_{n-i+1}.
\end{align*}
It follows that $f^{(n-i+1)}_{n-i}g^{(n-i+1)}_{n-i}\in \epsilon\O\setminus \epsilon^{2}\O$, and 
we may assume 
$$
 f^{(n-i+1)}_{n-i}=\epsilon, \;\; g^{(n-i+1)}_{n-i}=1,
$$
by swapping $E'$ and $E''$ if necessary. Thus, we have 
$$
\begin{pmatrix} \alpha_{n-i} & \beta_{n-i} \\ \gamma_{n-i} & \delta_{n-i} \end{pmatrix}
=\begin{pmatrix} \alpha_{n-i+1} & \epsilon^{-1}\beta_{n-i+1} \\ \epsilon\gamma_{n-i+1} & \delta_{n-i+1} \end{pmatrix}
=\cdots
=\begin{pmatrix} \alpha_{i} & \epsilon^{-1}\beta_{i} \\ \epsilon\gamma_{i} & \delta_{i} \end{pmatrix}.
$$
Finally, if $k=i+1$, then the similar argument shows 
$$ f^{(i+1)}_{i}g^{(i+1)}_{i}\in \epsilon\O\setminus \epsilon^{2}\O, $$ 
and we may assume that $(f^{(i+1)}_{i},g^{(i+1)}_{i})$ is either $(\epsilon,1)$ or $(1,\epsilon)$. In the former case, 
$$
\begin{pmatrix} \alpha_1 & \beta_1 \\ \gamma_1 & \delta_1 \end{pmatrix}
=\begin{pmatrix} \alpha_{i} & \epsilon^{-1}\beta_{i} \\ \epsilon\gamma_{i} & \delta_{i} \end{pmatrix}
=\begin{pmatrix} \epsilon^{-1}\alpha_{i+1} & \epsilon^{-1}\beta_{i+1} \\ \epsilon\gamma_{i+1} & \epsilon\delta_{i+1} \end{pmatrix},
$$
which implies that $\alpha_{i+1}, \beta_{i+1}\in \epsilon\O$, a contradiction. Thus, we obtain
$$
 f^{(i+1)}_{i}=1, \;\; g^{(i+1)}_{i}=\epsilon.
$$
Therefore, we have obtained the desired formula. In particular, we have the following.
\begin{gather*}
\alpha_{k-1}=\alpha_k,\;\; f^{(k)}_{k-1}\beta_{k-1}=g^{(k)}_{k-1}\beta_{k},\;\;
g^{(k)}_{k-1}\gamma_{k-1}=f^{(k)}_{k-1}\gamma_{k},\;\;\delta_{k-1}=\delta_k,\\
Xa_k=f^{(k)}_{k-1}a_{k-1},\;\; Xb_k=g^{(k)}_{k-1}b_{k-1}+\delta_{k,n}a_{n-i},
\end{gather*}
where $\delta_{k,n}$ is the Kronecker delta. Suppose that $1\leq k\leq n-1$. Then, we have 
\begin{align*}
XA'_{k}&=X(e'_k-\alpha_{k}a_k-\beta_{k}b_k)
=Xe'_k - f^{(k)}_{k-1}\alpha_{k}a_{k-1}-g^{(k)}_{k-1}\beta _{k}b_{k-1},\\
f^{(k)}_{k-1}A'_{k-1}&=f^{(k)}_{k-1}(e'_{k-1}-\alpha_{k-1}a_{k-1}-\beta_{k-1}b_{k-1})
=f^{(k)}_{k-1}e'_{k-1}-f^{(k)}_{k-1}\alpha_{k}a_{k-1}-g^{(k)}_{k-1}\beta_{k}b_{k-1}.
\end{align*}
We compute $Xe'_{k}-f^{(k)}_{k-1}e'_{k-1}$ in two ways:
\begin{align*}
Xe'_{k}-f^{(k)}_{k-1}e'_{k-1}&=XA'_{k}-f^{(k)}_{k-1}A'_{k-1} \in E'',\\
Xe'_{k}-f^{(k)}_{k-1}e'_{k-1}&= f^{(k)}_{k-2}e'_{k-2}+\cdots+f^{(k)}_{1}e'_{1} \in E'.
\end{align*}
Thus, we have $Xe'_{k}=f^{(k)}_{k-1}e'_{k-1}$, for $1\leq k\leq n-1$. Next suppose that $k=n$. Then, the similar computation shows
$$
\beta_{n}a_{n-i}+XA'_{n}-f^{(n)}_{n-1}A'_{n-1}=Xe'_{n}-f^{(n)}_{n-1}e'_{n-1}=f^{(n)}_{n-2}e'_{n-2}+\cdots+f^{(n)}_{1}e'_{1}.
$$
We compute $X^{n-i+1}e'_{n}-f^{(n)}_{n-1}X^{n-i}e'_{n-1}$ in two ways as before, and we obtain
$$
X^{n-i+1}A'_{n}-f^{(n)}_{n-1}X^{n-i}A'_{n-1}=f^{(n)}_{n-2}X^{n-i}e'_{n-2}+\cdots+f^{(n)}_{1}X^{n-i}e'_{1}=0.
$$
Hence, we have $f^{(n)}_{n-2}=\cdots=f^{(n)}_{n-i+1}=0$. We define
$$
z_{n}=e'_{n},\;\;
z_{k}=e'_{k}+X^{n-k-1}(f^{(n)}_{n-i}e'_{n-i}+\cdots+f^{(n)}_{1}e'_{1}),\;\text{for}\; 1\leq k\leq n-1.
$$
Then, $\{z_k \mid 1\leq k\leq n\}$ is an $\O$-basis of $E'$, since
$$ X^{n-k-1}(f^{(n)}_{n-i}e'_{n-i}+\cdots+f^{(n)}_{1}e'_{1})\in \Ker(X^{k-1}). $$
Further, we have $z_k=e'_k$, for $1\leq k\leq i-1$. In particular, $z_{n-i}=e'_{n-i}$ by $n-i\leq i-1$. 
Then, we can check that
\[Xz_{k}=\left\{\begin{array}{cl}
z_{k-1}&\mathrm{if}\ k\neq n-i+1,\\
\epsilon z_{k-1}&\mathrm{if}\ k=n-i+1.
\end{array}\right.\]
Thus, we conclude that $E'\simeq Z_{n-i}$. Recall that the exact sequence
\[0\rightarrow Z_{n-i}\rightarrow E_{i}\rightarrow Z_{i}\rightarrow 0 \]
% it implies that there is an irreducible morphism $Z_{n-i}\rightarrow Z_{i}$. Thus, $E_{i}\simeq Z_{n-i}\oplus Z_{i}$ and 
% the stable Auslander-Reiten quiver has the finite component $\{Z_{i},Z_{n-i}\}$. But it is impossible by
% Theorem \ref{infinite tree class theorem} and Lemma \ref{infinite lattices}. 
does not split. On the other hand, $E_{i}\simeq Z_{n-i}\oplus Z_{i}$ implies that it must split, by Miyata's theorem \cite[Thm.1]{M}. 
Hence, $E_i$ is indecomposable in (case a).

Next assume that we are in (case b). Then, $f^{(k)}_{k-1}$ and $g^{(k)}_{k-1}$, for $k\neq i+1$, are invertible as before, and we may choose 
$$ f^{(k)}_{k-1}=1,\quad g^{(k)}_{k-1}=1.$$

\noindent
If $k=i+1$, note that
$$
f^{(i+1)}_i\alpha_{i}=\epsilon \alpha_{i+1}, \;\;
f^{(i+1)}_i\beta_{i}=\epsilon \beta_{i+1}, \;\;
g^{(i+1)}_i\gamma_{i}=\epsilon \gamma_{i+1}, \;\;
g^{(i+1)}_i\delta_{i}=\epsilon \delta_{i+1}.
$$
Thus, $\alpha_i,\beta_i\in\epsilon\O$ if $f^{(i+1)}_i$ is invertible, and $\gamma_i,\delta_i\in\epsilon\O$ if $g^{(i+1)}_i$ is invertible. But
both are impossible. Further, 
$$
f^{(i+1)}_ig^{(i+1)}_i(\alpha_{i}\delta_{i}-\beta_{i}\gamma_{i})=\epsilon^2(\alpha_{i+1}\delta_{i+1}-\beta_{i+1}\gamma_{i+1})
$$
implies $f^{(i+1)}_ig^{(i+1)}_i\in \epsilon^2\O\setminus \epsilon^3\O$. Thus, we may choose
$$ f^{(i+1)}_i=\epsilon,\quad g^{(i+1)}_i=\epsilon. $$
Hence, we may assume without loss of generality that 
$$
f^{(k)}_{k-1}=g^{(k)}_{k-1}
=\left\{\begin{array}{cl}
1& \mathrm{if}\ k\neq i+1,\\
\epsilon & \mathrm{if}\ k=i+1.
\end{array}\right.
$$
$$
\begin{pmatrix} \alpha_1 & \beta_1 \\ \gamma_1 & \delta_1 \end{pmatrix}
=\cdots\cdot=\begin{pmatrix} \alpha_i & \beta_i \\ \gamma_i & \delta_i \end{pmatrix}
=\begin{pmatrix} \alpha_{i+1} & \beta_{i+1} \\ \gamma_{i+1} & \delta_{i+1} \end{pmatrix}
=\cdots\cdot=\begin{pmatrix} \alpha_n & \beta_n \\ \gamma_n & \delta_n \end{pmatrix}.
$$
and $Xa_k=f^{(k)}_{k-1}a_{k-1}$, $Xb_k=g^{(k)}_{k-1}b_{k-1}+\delta_{k,n}a_{i}$.
For $1\leq k\leq n-1$, we have
$$
XA'_k-f^{(k)}_{k-1}A'_{k-1}=Xe'_k-f^{(k)}_{k-1}e'_{k-1}=f^{(k)}_{k-2}e'_{k-2}+\cdots+f^{(k)}_1e'_1,
$$
and the same argument as before shows that
\[Xe'_{k}=\left\{\begin{array}{ll}
f^{(k)}_{k-1}e'_{k-1} &\mathrm{if}\ k\neq n,\\
f^{(n)}_{n-1}e'_{n-1}+f^{(n)}_{i}e'_{i}+
\cdots+f^{(n)}_{1}e'_{1} & \mathrm{if}\ k=n.
\end{array}\right.\]
Now, we compute
\begin{align*}
X^{i-1}e'_{n-1}&=f^{(n-1)}_{n-2}\cdots f^{(n-i+1)}_{n-i}e'_{n-i}=\epsilon e'_i,\\
X^{i}a_n&=f^{(n)}_{n-1}\cdots f^{(i+1)}_{i}a_i=\epsilon a_i,\\
X^{i}b_n&=X^{i-1}(b_{n-1}+a_i)=g^{(n-1)}_{n-2}\cdots g^{(i+1)}_{n-i}b_i+f^{(i)}_{i-1}\cdots f^{(2)}_{1}a_1=\epsilon b_i+a_1.
\end{align*}
Thus, we have
\begin{align*}
X^ie'_n-X^{i-1}e'_{n-1}&=X^{i}(\alpha_{n}a_n+\beta_{n}b_n+A'_n)-\epsilon e'_i\\
&=\epsilon(\alpha_{n}a_i+\beta_{n}b_i-e'_i)+X^{i}A'_n+\beta_{n}a_1.
\end{align*}
If $i+1\leq k\leq n-1$ then $k-i+1\leq n-i=i$ and
$$
X^ie'_k=f^{(k)}_{k-1}\cdots f^{(k-i+1)}_{k-i}e'_{k-i}\in \epsilon E'.
$$
Thus, $X^iA'_n\in \epsilon E'$. On the other hand, we have
\begin{align*}
X^ie'_n-X^{i-1}e'_{n-1}&=X^{i-1}(Xe'_n-e'_{n-1})=X^{i-1}(f^{(n)}_{i}e'_i+\cdots+f^{(n)}_1e'_1)\\
&=f^{(n)}_{i}X^{i-1}e'_i=f^{(n)}_{i}X^{i-1}(\alpha_{i}a_i+\beta_{i}b_i+A'_i)\\
&=f^{(n)}_{i}(\alpha_{i}X^{i-1}a_i+\beta_{i}X^{i-1}b_i)=f^{(n)}_{i}(\alpha_{i}a_1+\beta_{i}b_1).
\end{align*}
Hence, we obtain $\beta_{n}a_1\equiv f^{(n)}_{i}(\alpha_{i}a_1+\beta_{i}b_1)\mod\epsilon\O$. The similar computation using $e''_k$ shows 
$\delta_{n}a_1\equiv f^{(n)}_{i}(\gamma_{i}a_1+\delta_{i}b_1)\mod\epsilon\O$. If $f^{(n)}_{i}$ was invertible, it would imply 
$\beta_{i}, \delta_{i}\in \epsilon\O$, which contradicts $\alpha_{i}\delta_{i}-\beta_{i}\gamma_{i}\in\O^\times$. Thus, $f^{(n)}_{i}\in \epsilon\O$ and we have 
$\beta_{n}, \delta_{n}\in \epsilon\O$, which is again a contradiction. Hence, $E_i$ is indecomposable in (case b).

Finally, suppose that we are in (case c). Since $E_{i}\simeq\tau(E_{n-i})$, for $2\leq i\leq n-2$, and $E_{n-i}$ is indecomposable by (cases a), 
It follows from \cite[Chap. III, Prop.1.7, Prop.1.8]{A1} that $E_{i}$ is indecomposable in (case c).
\end{proof}

\subsection{Almost split sequence ending at $\mathbf{E_i}$}
We construct an almost split sequence ending at $E_{i}$, for $2\leq i\leq n-2$. Define $\pi: A^{\oplus 4}\to E_{i}$, for $2\leq i\leq n-2$, by
$$
\pi(p,q,r,s)=(\epsilon p+X^{i-1}q,X^{n-i}p,\epsilon q+\epsilon X r+X^{n-i }s), 
$$
for $(p,q,r,s)\in A^{\oplus 4}$. Note that 
$$ \pi(1,0,0,0)=a_n,\; \pi(0,1,0,0)=b_n,\; \pi(0,0,1,0)=b_{n-1},\; \pi(0,0,0,1)=b_i. $$

\begin{lemma}\label{kernel is E}
Let $\pi:A^{\oplus 4}\to E_i$ be as above. Then, 
\begin{itemize}
\item[(1)]
$\pi$ is an epimorphism.
\item[(2)]
$\Ker(\pi)\simeq E_{n-i}$, for $2\leq i\leq n-2$.
\end{itemize}
\end{lemma}
\begin{proof}
(1)\;  It is easy to check that $a_k, b_k\in \Im(\pi)$, for $1\leq k\leq n$. Note that $E_{i}$ is generated by
$\{a_{n},b_{n},b_{n-1},b_{i}\}$ as an $A$-module and $a_{n-i}=X b_{n}-b_{n-1}$. 

(2)\; We define an $A$-module homomorphism $\iota: E_{n-i} \to A^{\oplus 4}$ by
$$ \iota(f,g,h)=(g,-X f+\frac{X^{n-i} h}{\epsilon},f,-h), \;\;\text{for $(f,g,h)\in E_{n-i}$.} $$
We write $h=h_0\epsilon + h_1\epsilon X + \cdots + h_{i-1}\epsilon X^{i-1} + h_{i}X^{i} + \cdots + h_{n-1}X^{n-1}$, for $h_i \in\O$. Then, 
$$
\frac{X^{n-i} h}{\epsilon}=h_0 X^{n-i} + h_1 X^{n-i+1} + \cdots + h_{i-1} X^{n-1}.
$$
Note that $(f,g,h)\in A^{\oplus 3}$ belongs to $E_{n-i}$ if and only if $h\in Z_{n-i}$ and $X^{i}f-\epsilon g=h_0X^{n-1}$. 
It is clear that $\iota$ is a monomorphism and it suffices to show that $\Im(\iota)=\Ker(\pi)$. Since
\[\begin{array}{lll}
\pi\iota(f,g,h)&=&(\epsilon g-X^{i}f+\frac{X^{n-1}h}{\epsilon},X^{n-i}g,\epsilon(-X f+\frac{X^{n-i}h}{\epsilon})+\epsilon X f-X^{n-i}h)\\
&=&(\epsilon g-X^{i}f+\frac{X^{n-1}h}{\epsilon},X^{n-i}g,0)=(0,0,0),
\end{array}
\]
we have $\Im(\iota)\subseteq\Ker(\pi)$. Let $(p,q,r,s)\in \Ker(\pi)$. Then we have
\[\begin{array}{rcc}
\epsilon p+X^{i-1}q&=&0,\\
X^{n-i}p&=&0,\\
\epsilon q+\epsilon Xr+X^{n-i}s&=&0.\\
\end{array}
\]
The third equation shows that the projective cover $A\twoheadrightarrow M_{n-i}=X^{n-i}A+\epsilon A/\epsilon A\subseteq A\otimes\kappa$ given by 
$f\mapsto X^{n-i}f+\epsilon A$ sends $s$ to $0$. Thus, we have $s\in Z_{n-i}$. Further, 
$$ X^{n-1}s+\epsilon(-\epsilon p+X^{i}r)=X^{n-1}s+\epsilon(X^{i-1}q+X^{i}r)=X^{i-1}(X^{n-i}s+\epsilon q+Xr)=0$$
implies $X^{i}r-\epsilon p=\frac{X^{n-1}(-s)}{\epsilon}$. Hence, we have $(r,p,-s)\in E_{n-i}$ and
\[\iota(r,p,-s)=(p,-Xr-\frac{X^{n-i}s}{\epsilon},r,s)=(p,q,r,s).\]
Therefore, we have $\Ker(\pi)=\Im(\iota)$, which implies $\Ker(\pi)\simeq E_{n-i}$.
\end{proof}

We consider the following pullback diagram:

\bigskip
\hspace{20mm}
\begin{xy}
(0,10) *{0}="A1", (15,10) *{E_{n-i}}="B1", (30,10) *{F_i}="C1", (45,10) *{E_i}="D1", (60,10) *{0}="E1",
(0,0) *{0}="A2", (15,0) *{E_{n-i}}="B2", (30,0) *{A^{\oplus 4}}="C2", (45,0) *{E_i}="D2", (60,0) *{0}="E2",

\ar "A1";"B1"
\ar "B1";"C1"
\ar "C1";"D1"
\ar "D1";"E1"

\ar "A2";"B2"
\ar "B2";"C2"^\iota
\ar "C2";"D2"^{\;\;\pi}
\ar "D2";"E2"

\ar @{=} "B1";"B2"
\ar "C1";"C2"
\ar "D1";"D2"^{\phi}
\end{xy}

\bigskip
\noindent
where $\iota$ is the isomorphism $E_{n-i}\simeq \Ker(\pi)$ defined in the proof of Lemma \ref{kernel is E}, and 
\[\begin{array}{rll}
\phi(a_k)&=\;0 \quad &\text{for}\;\; 1\leq k\leq n,\\
\phi(b_k)&=\;0 \quad &\text{for}\;\; 1\leq k\leq n-1,\\
\phi(b_n)&=\;b_1 \quad &\text{for}\;\; k=n.
\end{array}\]

\begin{lemma}
\label{l1}
Suppose that $2\leq i\leq n-i$. 
Let $\rho\in \Rad \Endhom_{A}(E_{i})$ such that
$$ \rho(a_{n})=\alpha a_{n}+\beta b_{n}+A, \quad \rho(b_{n})=\alpha' a_{n}+\beta' b_{n}+B, $$
where $\alpha, \beta, \alpha', \beta'\in \O$ and $A, B\in\Ker(X^{n-1})$. Then we have the following. 
\begin{itemize}
\item[(1)] $\beta\in \epsilon\O$, and $\alpha\in\epsilon \O$ if and only if $\beta'\in \epsilon \O$.
\item[(2)] $\alpha\beta'-\beta\alpha'\in \epsilon \O$.
\end{itemize}
\end{lemma}
\begin{proof}
(1)\; We compute $\rho(\epsilon X^{n-i}b_n-X^{n-1}a_n)$ in two ways. Since $X^{n-i}b_n=\epsilon b_i+a_1$ and $X^{n-1}a_n=\epsilon a_1$, we have
$\rho(\epsilon X^{n-i}b_n-X^{n-1}a_n)=\epsilon^2\rho(b_i)\in \epsilon^2 E_i$. On the other hand, since $X^{n-i}b_n=\epsilon b_i+a_1$, we have
\begin{align*}
\rho(\epsilon X^{n-i}b_n-X^{n-1}a_n)&=\epsilon X^{n-i}(\alpha'a_n+\beta'b_n+B)-X^{n-1}(\alpha a_n+\beta b_n+A)\\
&=\epsilon\alpha'X^{n-i}a_n+\epsilon^2\beta'b_i+\epsilon(\beta'-\alpha)a_1-\epsilon\beta b_1+\epsilon X^{n-i}B.
\end{align*}
Then, $X^{n-i}a_k = \epsilon a_{k-n+i}$ and  $X^{n-i}b_k = \epsilon b_{k-n+i}$, for $n-i+1\leq k \leq n-1$, imply that 
$\epsilon X^{n-i}B\in \epsilon^2E_i$. Hence, we may divide the both sides by $\epsilon$. Reducing modulo $\epsilon$, we have
$$ (\beta'-\alpha)a_1-\beta b_1\equiv 0\; \mod\epsilon E_i, $$
since $X^{n-i}a_n\equiv 0 \mod \epsilon E_i$ if $2\leq i\leq n-i$.
Now, the claim is clear.

(2)\;Since $\rho(a_k),\rho(b_k)\in \Ker(X^k)$, we may write
\begin{align*}
\rho(a_{k})&=\alpha_{k}a_{k}+\beta_{k}b_{k}+A_{k},\\
\rho(b_{k})&=\alpha'_{k}a_{k}+\beta'_{k}b_{k}+B_{k},
\end{align*}
where $\alpha_{k},\beta_{k},\alpha'_{k},\beta'_{k}\in \O$ and $A_{k},B_{k}\in \Ker(X^{k-1})$. 
We claim that 
$$ \alpha_{k}\beta_{k}'-\beta_{k}\alpha_{k}'=\alpha\beta'-\beta\alpha'. $$
To see this, observe that we have the following identities in $E_i/\Ker(X^{k-1})$. 
\[\left\{\begin{array}{cll}
\alpha a_{k}+\beta b_{k} \equiv \rho(X^{n-k}a_{n}) \equiv \rho(a_{k}) &\mod \Ker(X^{k-1}) & \mathrm{if}\ k>n-i,\\
\alpha \epsilon a_{k}+\beta b_{k} \equiv \rho(X^{n-k}a_{n}) \equiv \epsilon \rho(a_{k}) &\mod \Ker(X^{k-1})& \mathrm{if}\ i< k\leq n-i,\\
\alpha \epsilon a_{k}+\beta \epsilon b_{k} \equiv \rho(X^{n-k}a_{n}) \equiv \epsilon \rho(a_{k}) &\mod \Ker(X^{k-1})& \mathrm{if}\ k\leq i ,
\end{array}\right.\]
\[\left\{\begin{array}{cll}
\alpha' a_{k}+\beta' b_{k} \equiv \rho(X^{n-k}b_{n}) \equiv \rho(b_{k}) &\mod \Ker(X^{k-1})& \mathrm{if}\ k>n-i,\\
\alpha' \epsilon a_{k}+\beta' b_{k} \equiv \rho(X^{n-k}b_{n}) \equiv \rho(b_{k}) &\mod \Ker(X^{k-1})& \mathrm{if}\ i< k\leq n-i,\\
\alpha' \epsilon a_{k}+\beta'\epsilon b_{k} \equiv \rho(X^{n-k}b_{n}) \equiv \epsilon \rho( b_{k}) &\mod \Ker(X^{k-1})& \mathrm{if}\ k\leq i.
\end{array}\right.\]

\medskip
\noindent
Thus, if we denote
\begin{align*}
(\overline{a}_k, \overline{b}_k)&=(a_{k}+\Ker(X^{k-1}),b_{k}+\Ker(X^{k-1})),\\
(\overline{a}'_k, \overline{b}'_k)&=(\rho(a_{k})+\Ker(X^{k-1}),\rho(b_{k})+\Ker(X^{k-1})),
\end{align*}
Then, we have
$$
(\overline{a}_k, \overline{b}_k)
\begin{pmatrix}
\alpha_k & \alpha'_k \\
\beta_k & \beta'_k
\end{pmatrix}
=(\overline{a}'_k, \overline{b}'_k)
=(\overline{a}_k, \overline{b}_k)
\begin{pmatrix}
\alpha & \alpha' \\
\beta & \beta'
\end{pmatrix}
\;\;\text{or}\;\;
(\overline{a}_k, \overline{b}_k)
\begin{pmatrix}
\alpha & \alpha'\epsilon \\
\beta\epsilon^{-1} & \beta'
\end{pmatrix}.
$$
Therefore, we have $\alpha_{k}\beta_{k}'-\beta_{k}\alpha_{k}'=\alpha\beta'-
\beta\alpha'$. In particular if $\alpha\beta'-\beta\alpha'\in \O^\times$, then $\rho$ is surjective, which contradicts $\rho\in \Rad \End_A(E_i)$.
\end{proof}

\begin{lemma}
\label{e-ass}
Suppose that $2\leq i\leq n-i$, and let $\phi\in \End_A(E_i)$ be as in the definition of the pullback diagram. Then we have the following.
\begin{itemize}
\item[(1)] $\phi$ does not factor through $\pi$.
\item[(2)] For any $\rho\in \Rad\Endhom_{A}(E_{i})$, $\phi\rho$ factors through $\pi$.
\end{itemize}
\end{lemma}
\begin{proof}
(1)\;Suppose that there exists 
$$ \psi=(\psi_1,\psi_2,\psi_3,\psi_4):E_{i}\longrightarrow A\oplus A\oplus A\oplus A $$
such that $\pi\psi=\phi$. Then, we have
\begin{align*}
0&=\pi\psi(a_{n})
=(\epsilon\psi_1(a_{n})+X^{i-1}\psi_2(a_{n}),X^{n-i}\psi_1(a_{n}),\epsilon \psi_2(a_{n})+\epsilon X\psi_3(a_{n})+X^{n-i}\psi_4(a_{n})),\\
b_1&=\pi\psi(b_{n})
=(\epsilon\psi_1(b_{n})+X^{i-1}\psi_2(b_{n}),X^{n-i}\psi_1(b_{n}),\epsilon \psi_2(b_{n})+\epsilon X\psi_3(b_{n})+X^{n-i}\psi_4(b_{n})).
\end{align*}
The first equality implies $\psi_4(X^{n-1}a_{n})\in \epsilon^{2}A$ by the following computation.
\begin{align*}
\psi_4(X^{n-1}a_{n})&=X^{i-1}( X^{n-i}\psi_4(a_{n}))=-X^{i-1}(\epsilon \psi_2(a_{n})+\epsilon X\psi_3(a_{n}))\\
&=-\epsilon X^{i-1}\psi_2(a_{n})-\epsilon \psi_3(X^ia_{n})=\epsilon^{2}\psi_1(a_{n})-\epsilon^{2}\psi_3(a_{n-i}).
\end{align*}
Thus, we conclude $\psi_4(X^{n-i}b_n)\equiv 0 \mod \epsilon A$ from
\begin{align*}
\epsilon\psi_4(X^{n-i}b_n)&=\epsilon\psi_4(X^{n-i-1}a_{n-i}+X^{n-i-1}b_{n-1})=\epsilon\psi_4(a_{1}+\epsilon b_{i})\\
&=\psi_4(\epsilon a_{1})+\epsilon^2\psi_4(b_{i})=\psi(X^{n-1}a_{n})+\epsilon^2\psi_4(b_{i})\in \epsilon^2A.
\end{align*}
On the other hand, using $b_1=(0,0,X^{n-1})$, the second equality implies 
\[\epsilon \psi_2(b_{n})+\epsilon X\psi_3(b_{n})+X^{n-i}\psi_4(b_{n})=X^{n-1},\]
and we have $\psi_4(X^{n-i}b_{n})\not\equiv 0 \mod \epsilon A$. Hence, we have reached a contradiction.

(2)\;Let $\rho\in\Rad\Endhom_{A}(E_{i})$. We write $\rho(a_{n})=\alpha a_{n}+\beta b_{n}+A$ and $\rho(b_{n})=\alpha' a_{n}+\beta' b_{n}+B$, 
where $\alpha, \beta, \alpha', \beta'\in\O$ and $A,B\in\Ker(X^{n-1})$. Then, $\phi\rho(a_{n})=\beta b_{1}$ and $\phi\rho(b_{n})=\beta' b_{1}$. 

By Lemma\:\ref{l1}(1), $\beta\in \epsilon\O$ and if $\beta'$ was invertible then $\alpha$ would be invertible, which contradicts Lemma\:\ref{l1}(2). 
Thus, $\beta,\beta'\in \epsilon \O$ and we may define $\psi_2:E_{i}\rightarrow A$ by
$$
(f,g,h) \mapsto \frac{\beta X^{n-1}f}{\epsilon^2}+\frac{\beta'X^{n-1}h}{\epsilon^2},
$$
where $(f,g, h)\in A\oplus A\oplus Z_i$ with $X^{n-i}f-\epsilon g=X^{n-1}h/\epsilon$. This is well-defined. Indeed, we have 
$\psi_2(a_k)=0$ and $\psi_2(b_k)=0$, for $1\leq k\leq n-1$, and 
$$
\psi_2(a_n)=\frac{\beta}{\epsilon}X^{n-1}, \quad \psi_2(b_n)=\frac{\beta'}{\epsilon}X^{n-1}. 
$$
Then $\psi=(0,\psi_2,0,0):E_{i}\rightarrow A\oplus A\oplus A\oplus A$ satisfies $\pi\psi=(X^{i-1}\psi_2,0,\epsilon \psi_2)=\phi\rho$.
\end{proof}

By Proposition \ref{construction of AR-seq} and Lemma \ref{e-ass}, we have an almost split sequence
\[0\rightarrow E_{n-i}\rightarrow F_{i}\rightarrow E_{i}\rightarrow 0, \]
where $F_{i}=\{(p,q,r,s,t)\in A^{\oplus 4}\oplus E_{i}\mid \pi (p,q,r,s)=\phi(t)\}$, for $2\leq i\leq n-i$. 

We define $z_k=(0,0,0,0,a_{k})\in F_{i}$, for $1\leq k\leq n$, and $x_{k}, y_{k}, w_{k}\in F_{i}$, for $1\leq k\leq n$, by

\[\begin{array}{lll}
x_{k}&=&\left\{\begin{array}{ll}
         (0,0,0,X^{n-k},a_{k})& \mathrm{if}\ 1\leq k\leq n-i,\\
         (0,0,-X^{2n-i-k-1},\epsilon X^{n-k},a_{k})& \mathrm{if}\ n-i<k \leq n.\\
\end{array}\right.\\[10pt]
y_{k}&=&\left\{\begin{array}{ll}
         (0,0,0,0,b_{k})& \mathrm{if}\ 1\leq k\leq i,\\
         (0,0,0,X^{n+i-k-1},b_{k}+a_{k-i+1})& \mathrm{if}\ i<k<n,\\
         (0,0,0,X^{i-1},b_{n})& \mathrm{if}\ k=n.\\
\end{array}\right.\\[20pt]
w_{k}&=&\left\{\begin{array}{ll}
         (0,-X^{n-k+1},X^{n-k},0,0)& \mathrm{if}\ 1\leq  k\leq i,\\
         (X^{n-k+i},-\epsilon X^{n-k+1},\epsilon X^{n-k},0,0)& \mathrm{if}\ i<k\leq n.\\
\end{array}\right.\\[10pt]
        \end{array}             
\]

\medskip
\noindent
Note that $(p,q,r,s,t)\in F_i$ if and only if 
$$
(\epsilon p+X^{i-1}q,X^{n-i}p, \epsilon q+\epsilon Xr+X^{n-i}s)=\beta_n b_1,
$$
where $t=\sum_{k=1}^n (\alpha_ka_k+\beta_kb_k)$.

\begin{lemma}
\label{F-basis}
$\{x_{k}, y_{k}, z_{k}, w_{k} \mid 1\leq k\leq n \}$ is an $\O$-basis of $F_{i}$. 
\end{lemma}
\begin{proof}
It suffices to show that they generate $F_i$ as an $\O$-module, since ${\rm rank}\;F_i=4n$. Let $F'_i$ be the 
$\O$-submodule generated by $\{x_{k}, y_{k}, z_{k}, w_{k} \mid 1\leq k\leq n \}$. We show first 
that $(\Ker(\pi),0)\subseteq F'_i$. Recall that any element of $(\Ker(\pi),0)=(\Im(\iota),0)$ has the form
$$
(g,-X f+\frac{X^{n-i} h}{\epsilon},f,-h,0),
$$
where $(f,g,h)\in A\oplus A\oplus Z_{n-i}$ and $X^{i}f-\epsilon g=X^{n-1}h/\epsilon$. Thus, $X^{n-i}g=0$ and $g$ is an $\O$-linear 
combination of $X^{n-k+i}$, for $i<k\leq n$. Thus, subtracting the corresponding $\O$-linear combination of $w_k$, for $i<k\leq n$, 
we may assume $g=0$. Since
$$ h\in Z_{n-i}=\O\epsilon\oplus\cdots\oplus \O\epsilon X^{i-1}\oplus \O X^i\oplus\cdots\oplus \O X^{n-1}, $$
we may further subtract an $\O$-linear combination of $x_k$, for $1\leq k\leq n$, and we may asume $g=h=0$ without loss of generality. Then, 
$(0,-Xf,f,0,0)$, for $f\in A$ with $X^if=0$, is an $\O$-linear combination of $w_k$, for $1\leq k\leq i$. 
Hence, $(\Ker(\pi),0)\subseteq F'_i$. Next we show that $(0,0,0,0,\Ker(\phi))\subseteq F'$. But it is clear from 
$(0,0,0,0,a_k)=z_k$, for $1\leq k\leq n$, and
\[
(0,0,0,0,b_k)=
\left\{\begin{array}{ll}
y_k \quad &\text{if $1\leq k\leq i$}, \\
y_k-x_{k-i+1} \quad &\text{if $i<k< n$}. 
\end{array}\right.
\]
Suppose that $(p,q,r,s,t)\in F_{i}$. Write $t=\beta b_n+t'$ such that $\beta\in\O$ and $t'\in \Ker(\phi)$. Then, to show that 
$(p,q,r,s,t)\in F'_{i}$, it is enough to see $(p,q,r,s,\beta b_n)\in F'_i$. Since
$$ \epsilon q+\epsilon Xr+X^{n-i}s=\beta X^{n-1}, $$
we have $(p,q,r,s-\beta X^{i-1})\in \Ker(\pi)$. Therefore, we deduce
$$ (p,q,r,s,\beta b_n)=(p,q,r,s-\beta X^{i-1},0)+\beta(0,0,0,X^{i-1},b_n)\in F'_i, $$
because $(0,0,0,X^{i-1},b_n)=y_n$.
\end{proof}

Let $F'_i$ be the $\O$-span of $\{ x_k, y_k, w_k \mid 1\leq k\leq n\}$, $F''_i$ the $\O$-span of $\{ z_k \mid 1\leq k\leq n\}$. 
It is easy to compute as follows.
\[\begin{array}{lll}
Xw_{k}&=&\left\{\begin{array}{ll}
         w_{k-1} \quad & \mathrm{if}\ k\neq i+1,  \\
         \epsilon w_{i} & \mathrm{if}\ k=i+1.\\
\end{array}\right.\\[10pt]
Xx_{k}&=&\left\{\begin{array}{ll}
          x_{k-1} & \mathrm{if}\ k\neq n-i+1,\\
          \epsilon x_{n-i}-w_{1}  & \mathrm{if}\ k=n-i+1.\\
\end{array}\right.\\[10pt]
Xy_{k}&=&\left\{\begin{array}{ll}
         y_{k-1}\quad & \mathrm{if}\ k\neq i+1, \\
         \epsilon y_{i}+x_{1}  & \mathrm{if}\ k=i+1.\\
         \end{array}\right.\\[10pt]
Xz_{k}&=&\left\{\begin{array}{ll}
         z_{k-1} \quad & \mathrm{if}\ k\neq  n-i+1,\\
         \epsilon z_{n-i} & \mathrm{if}\ k=n-i+1.\\
        \end{array}\right.\\[10pt]
\end{array}        
\]
Hence, the direct summands $F_{i}'$ and $F_{i}''$ of $F_{i}= F_{i}'\oplus F_{i}''$ are $A$-lattices and $F_{i}''\simeq Z_{n-i}$.

\begin{lemma}
\label{F-indec}
The middle term of the almost split sequence ending at $E_i$, for $2\leq i\leq n-2$, is the direct sum of 
$Z_{n-i}$ and an indecomposable direct summand.
\end{lemma}        
\begin{proof}
Since $\tau(Z_i)\simeq Z_{n-i}$ implies $\tau(E_i)\simeq E_{n-i}$, we may assume $2\leq i \leq n-i$ without loss of generality. 
Let $F'_{i}$ be the $A$-lattice as above. Then we have to show that $F'_{i}$ is an indecomposable $A$-lattice. 
Suppose that $F'_{i}$ is not indecomposable. Then, there exist $A$-sublattices $Z$ and $L$ such that
$F'_{i}\simeq Z \oplus L$ and $Z\otimes\K\simeq A\otimes\K$. Since
$$ \Ker(X^{k})\cap F'_{i}=\bigoplus_{1\leq j\leq k}(\O w_{j}+\O x_{j}+\O y_{j}), $$
we may choose an $\O$-basis $\{e_{k} \mid 1\leq k\leq n \}$ of $Z$ such that 
$$ e_{k}=\alpha_{k}w_{k}+\beta_{k}x_{k}+\gamma_{k}y_{k}+A_{k}, $$
where $\alpha_k, \beta_k, \gamma_k\in \O$ with $(\alpha_{k},\beta_{k},\gamma_{k})\not\in (\epsilon\O)^{\oplus 3}$ and $A_{k}\in \Ker(X^{k-1})\cap L$. 
Then, 
$$ \Ker(X^{k})\cap Z=\O e_1\oplus \cdots \oplus \O e_k $$
and at least one of $\alpha_{k}, \beta_{k}, \gamma_{k}$ is invertible. Write 
$$ Xe_{k}=f^{(k)}_{k-1}e_{k-1}+\cdots+f^{(k)}_{1}e_{1}, $$
for $f^{(k)}_1,\dots, f^{(k)}_{k-1}\in \O$. We first assume that $2\leq i < n-i$. Note that
 \[X e_{k}=\left\{\begin{array}{ll}
         \alpha_{k}w_{k-1}+\beta_{k}x_{k-1}+\gamma_{k}y_{k-1}+XA_{k} & \mathrm{if}\ k\neq i+1,  n-i+1,\\
         \alpha_{n-i+1}w_{n-i}+\beta_{n-i+1}(\epsilon x_{n-i}-w_{1})+\gamma_{n-i+1}y_{n-i} +XA_{n-i+1} & \mathrm{if}\ k=n-i+1,\\ 
         \alpha_{i+1}\epsilon w_{i}+\beta_{i+1}x_{i}+\gamma_{i+1}(\epsilon y_{i}+x_{1}) +XA_{i+1} 
         & \mathrm{if}\ k=i+1.\\ 
        \end{array}\right.\]
Thus, we have 
\[f^{(k)}_{k-1}(\alpha_{k-1},\beta_{k-1},\gamma_{k-1})=
\left\{\begin{array}{ll}
         (\alpha_{k},\beta_{k},\gamma_{k})& \mathrm{if}\ k\neq i+1,  n-i+1,\\
         (\alpha_{n-i+1},\epsilon\beta_{n-i+1},\gamma_{n-i+1}) & \mathrm{if}\ k=n-i+1,\\ 
         (\epsilon \alpha_{i+1},\beta_{i+1},\epsilon\gamma_{i+1}) 
         & \mathrm{if}\ k=i+1.\\ 
        \end{array}\right.\] 
We may assume one of the following two cases occurs. 
\begin{enumerate}
\item $f^{(k)}_{k-1}=1\ (k\neq n-i+1)$, $f^{(n-i+1)}_{n-i}=\epsilon$.
\item $f^{(k)}_{k-1}=1\ (k\neq i+1)$, $f^{(i+1)}_{i}=\epsilon$.
\end{enumerate}
In fact, since at least one of $\alpha_{k}, \beta_{k}, \gamma_{k}$ is invertible, if $k\neq n-i+1, i+1$ then 
$f^{(k)}_{k-1}$ is invertible. We multiply its inverse to $e_k$, and we obtain
$$
f^{(2)}_1=\cdots=f^{(i)}_{i-1}=1 \;\; \text{and}\;\; (\alpha_1,\beta_1,\gamma_1)=\cdots=(\alpha_{i},\beta_{i},\gamma_{i})
$$
in the new basis. By the same reason, we have $f_{k-1}^{(k)}\not\in \epsilon^{2}\O$, for all $k$. Suppose that 
both $f^{(n-i+1)}_{n-i}$ and $f^{(i+1)}_{i}$ are invertible. Then, we may reach
$$
(\alpha_i,\beta_i,\gamma_i)=(\epsilon\alpha_{i+1},\beta_{i+1},\epsilon\gamma_{i+1})=\cdots
=(\epsilon\alpha_{n-i},\beta_{n-i},\epsilon\gamma_{n-i})=(\epsilon\alpha_{n-i+1},\epsilon\beta_{n-i+1},\epsilon\gamma_{n-i+1}),
$$
which is a contradiction. Suppose that both $f^{(n-i+1)}_{n-i}$ and $f^{(i+1)}_{i}$ are not invertible. Then,
\begin{align*}
(\alpha_i,\beta_i,\gamma_i)&=(\alpha_{i+1},\epsilon^{-1}\beta_{i+1},\gamma_{i+1})=\cdots=(\alpha_{n-i},\epsilon^{-1}\beta_{n-i},\gamma_{n-i})\\
&=(\epsilon^{-1}\alpha_{n-i+1},\epsilon^{-1}\beta_{n-i+1},\epsilon^{-1}\gamma_{n-i+1}),
\end{align*}
which implies that none of $\alpha_{n-i+1}, \beta_{n-i+1}, \gamma_{n-i+1}$ is invertible. Thus, we have proved that we are in 
case (1) or case (2). Suppose that we are in case (1). Then, we have
\[\begin{array}{lll}Xe_{k}-f^{(k)}_{k-1}e_{k-1}&=&f^{(k)}_{k-2}e_{k-2}+\cdots+f^{(k)}_{1}e_{1}\\[5pt]
&=&\left\{\begin{array}{ll}
XA_{k}-A_{k-1}& \mathrm{if}\ k\neq n-i+1, i+1,\\
XA_{n-i+1}-\epsilon A_{n-i}-\beta_{n-i+1}w_{1}& \mathrm{if}\ k=n-i+1,\\
XA_{i+1}-A_{i}+\gamma_{i+1}x_{1}& \mathrm{if}\ k=i+1.\\
\end{array}\right.\\
\end{array}
\]
Since $A_{k}\in \Ker(X^{k})\cap L$, we obtain that
\[Xe_{k}=\left\{\begin{array}{ll}
e_{k-1}& \mathrm{if}\ k\neq n-i+1,i+1,\\
\epsilon e_{n-i}+f_{1}^{(n-i+1)}e_{1} & \mathrm{if}\ k=n-i+1,\\
e_{i}+f^{(i+1)}_{1}e_{1} & \mathrm{if}\ k=i+1,\\
\end{array}\right.\\ 
\]
and $XA_{n-i+1}=X^{2}A_{n-i+2}=\cdots=X^{i}A_{n}$. As we are in case (1), 
\begin{align*}
(\alpha_1,\beta_1,\gamma_1)&=(\alpha_2,\beta_2,\gamma_2)=\cdots=(\alpha_{i},\beta_{i},\gamma_{i}) \\
&=(\epsilon\alpha_{i+1},\beta_{i+1},\epsilon\gamma_{i+1})=\cdots =(\epsilon\alpha_{n-i},\beta_{n-i},\epsilon\gamma_{n-i}) \\
&=(\alpha_{n-i+1},\beta_{n-i+1},\gamma_{n-i+1})=\cdots=(\alpha_n,\beta_n,\gamma_n),
\end{align*}
so that we may write
\[e_{k}=\left\{\begin{array}{ll}
\epsilon\alpha w_k + \beta x_k + \epsilon\gamma y_k + A_k &  \mathrm{if}\ 1\leq k\leq i \; \text{or}\; n-i+1\leq k\leq n,\\
\alpha w_k + \beta x_k + \gamma y_k + A_k & \mathrm{if}\ i+1 \leq k \leq n-i,\\
\end{array}\right.\\[5pt]
\]
with $\alpha, \gamma\in \O$ and $\beta\in\O^\times$. Then, $Xe_{n-i+1}=\epsilon e_{n-i}+f_1^{(n-i+1)}e_1$ implies
$$
\epsilon\alpha w_{n-i}+\beta(\epsilon x_{n-i}-w_1)+\epsilon\gamma y_{n-i}+X^iA_n
=\epsilon e_{n-i}+f_1^{(n-i+1)}(\epsilon\alpha w_1+\beta x_1+\epsilon\gamma y_1).
$$
We equate the coefficients of $w_1$ on both sides. Since contribution from $X^iA_n$ comes from $X^iw_{i+1}=\epsilon w_1$ only, 
we conclude that $\beta\in \epsilon\O$, which contradicts $\beta\in\O^\times$.

Suppose that we are in case (2). Then, the same argument as above shows that
\[Xe_{k}=\left\{\begin{array}{ll}
e_{k-1}  &  \mathrm{if}\ k\neq n-i+1,i+1,\\
e_{n-i}+f_{1}^{(n-i+1)}e_{1} & \mathrm{if}\ k=n-i+1,\\
\epsilon e_{i}+f^{(i+1)}_{1}e_{1} & \mathrm{if}\ k=i+1.\\
\end{array}\right.\\[5pt]
\]
We define an $\O$-basis $\{e''_{k}\}$ of $Z$ as follows:
\begin{itemize}
\item[(i)] $e''_{k}=e_{k}\ (1\leq k \leq i)$.
\item[(ii)] $e''_{n-i}=e_{n-i}-f_{1}^{(i+1)}e_{n-2i+1}+f_{1}^{(n-i+1)}e_{1}$.
\item[(iii)] $e''_{n-1}=e_{n-1}-f_{1}^{(i+1)}e_{n-i}-f_{1}^{(i+1)}f_{1}^{(n-i+1)}e_1$.
\item[(iv)] $e''_{k}=e_{k}-f_{1}^{(i+1)}e_{k-i+1}\ (i+1 \leq k\leq n,\ k\neq n-i,\,n-1)$.
\end{itemize}
Then, we have $Z\simeq Z_{i}$. To summarize, we have proved that if there is a direct summand of rank $n$ then it must be 
isomorphic to $Z_{i}$. As there is an irreducible morphism $Z_{i}\rightarrow E_{i}$, $E_{i}$ must be a direct summand 
of $E_{n-i}$ and we conclude $E_{i}\simeq E_{n-i}$. Then there exist $a'_k, b'_k\in E_{n-i}$, for $1\leq k\leq n$, such that
\begin{align*}
a_n&=\alpha a'_n+\beta b'_n +A,\\
b_n&=\gamma a'_n+\delta b'_n +B,
\end{align*}
where $\alpha, \beta, \gamma, \delta\in \O$ with $\alpha\delta-\beta\gamma\in \O^\times$, $A,B\in \Ker(X^{n-1})$, and 
\begin{align*}
Xa'_k&=\left\{\begin{array}{ll}
a'_{k-1}& (k\neq n-i+1)\\
\epsilon a'_{k-1}& (k=n-i+1),\\
\end{array}\right.\\
Xb'_k&=\left\{\begin{array}{ll}
b'_{k-1}& (k\neq i+1,n)\\
\epsilon b'_{k-1}& (k=i+1)\\
a'_{n-i}+b'_{n-1}& (k=n).\\
\end{array}\right.
\end{align*}
We compute $X^{n-i}a_n$ and $X^{n-i}b_n$ as follows. 
\begin{align*}
\epsilon a_{i}&=\epsilon(\alpha a'_{i}+\beta b'_{i})+\beta a'_1+X^{n-i}A,\\
\epsilon b_{i}&=\epsilon(\gamma a'_{i}+\delta b'_{i})+\delta a'_1+X^{n-i}B.
\end{align*}
Since $X^{n-i}A, X^{n-i}B\in \epsilon E_{n-i}$ by $2\leq i< n-i$, we have $\beta,\delta\in \epsilon\O$, which is a contradiction. 

Thus, $F'_{i}$ is indecomposable if $2\leq i<n-i$. It remains to consider $2\leq i=n-i$. We choose an $\O$-basis $\{ e_k \mid 1\leq k\leq n\}$ of $Z$ and write 
$$ e_k=\alpha_{k}w_k+\beta_{k}x_k+\gamma_{k}y_k + A_k, $$
as before. Then, we have
\[X e_{k}=\left\{\begin{array}{ll}
         \alpha_{k}w_{k-1}+\beta_{k}x_{k-1}+\gamma_{k}y_{k-1}+XA_{k} & \mathrm{if}\ k\neq i+1,\\
         \alpha_{i+1}\epsilon w_{i}+\beta_{i+1}(\epsilon x_{i}-w_{1})+\gamma_{i+1}(\epsilon y_{i}+x_{1}) +XA_{i+1} 
         & \mathrm{if}\ k=i+1,\\ 
        \end{array}\right.\]
and it follows that
\[f^{(k)}_{k-1}(\alpha_{k-1},\beta_{k-1},\gamma_{k-1})=
\left\{\begin{array}{ll}
         (\alpha_{k},\beta_{k},\gamma_{k})& \mathrm{if}\ k\neq i+1, \\
   (\epsilon \alpha_{i+1},\epsilon\beta_{i+1},\epsilon\gamma_{i+1}) 
         & \mathrm{if}\ k=i+1.\\ 
        \end{array}\right.\] 
Hence, we may assume $f^{(k)}_{k-1}=1$, for $k\neq i+1$, and $f^{(i+1)}_{i}=\epsilon$, without loss of generality.
Since $A_{k}\in \Ker(X^{k-1})\cap L $, we obtain from the computation of $Xe_k-f^{(k)}_{k-1}e_{k-1}$ that
\[Xe_{k}=\left\{\begin{array}{ll}
e_{k-1}& \mathrm{if}\ k\neq i+1,\\
\epsilon e_{i}+f^{(i+1)}_{1}e_{1}& \mathrm{if}\ k=i+1,\\
\end{array}\right.\\ 
\]
and $XA_{i+1}=X^2A_{i+2} =\cdots=X^iA_{n}$. Let $\lambda$, $\mu$, $\nu$ be the
coefficient of $w_{n-i+1}$, $x_{n-i+1}$, $y_{n-i+1}$ in $A_{n}$, respectively. Then the coefficient of $w_1$, $x_1$, $y_1$ in 
$XA_{i+1}$ are $\epsilon \lambda$, $\epsilon\mu$, $\epsilon\nu$. Since
$f^{(i+1)}_{1}e_1=XA_{i+1}-\epsilon A_{i}-\beta_{i+1}w_1+\gamma_{i+1}x_1$, we have
$$
f^{(i+1)}_1\alpha_1\equiv -\beta_{i+1} \mod \epsilon\O, \;\; f^{(i+1)}_1\beta_1\equiv \gamma_{i+1} \mod \epsilon\O, \;\;
f^{(i+1)}_1\gamma_1\equiv 0 \mod \epsilon\O. 
$$
We may show that $f^{(i+1)}_{1}$ is not invertible, but whenever it is invertible or not, 
$$
\gamma_{1}=\gamma_{2}=\cdots=\gamma_{n}\;\;\text{and}\;\;
\beta_{1}=\beta_{2}=\cdots=\beta_{n}
$$
imply that $\beta_k\equiv 0\mod \epsilon\O$ and $\gamma_k\equiv 0\mod \epsilon\O$, for $1\leq k\leq n$. 
It follows that we may choose an $\O$-basis $\{ a'_{k}, b'_{k} \mid 1\leq k\leq n \}$ of $L$ in the following form. 
\begin{align*}
a'_{k}&= \lambda'_{k}w_{k}+x_{k}+A'_{k}, \\
b'_{k}&= \lambda''_{k}w_{k}+y_{k}+B'_{k},
\end{align*}
where $\lambda', \lambda''\in \O$ and $A'_{k},B'_{k}\in \Ker(X^{k-1})\cap Z$. Write
$$ Xa'_{k}=\sum_{j=1}^{k-1}(g^{(k)}_{j}a'_{j}+h^{(k)}_{j}b'_{j}). $$
Multiplying $a'_{k}= \lambda'_{k}w_{k}+x_{k}+A'_{k}$ with $X$, we obtain
\[Xa'_{k}=\left\{
\begin{array}{ll}
\lambda_{k}'w_{k-1}+x_{k-1}+XA'_{k}&\mathrm{if}\ k\neq i+1,\\
\epsilon\lambda_{i+1}'w_{i}+\epsilon x_{i}-w_{1}+XA'_{i+1}& \mathrm{if\ }k=i+1.\\
\end{array}\right.
\]
Thus, $g_{k-1}^{(k)}=1$, for $k\neq i+1$, $g_{i}^{(i+1)}=\epsilon$, and $h_{k-1}^{(k)}=0$, for all $k$. Further, we have
\[Xa_{k}'-g_{k-1}^{(k)}a'_{k-1}=
\left\{\begin{array}{ll}
XA_{k}'-A_{k-1}'& \mathrm{if}\ k\neq i+1,\\
XA_{i+1}'-\epsilon A_{i}'-w_{1}& \mathrm{if}\ k=i+1.
\end{array}\right.\]
We obtain $Xa'_{k}-a'_{k-1}=0$ if $k\neq i+1$, and if $k=i+1$ then $Xa'_{i+1}-\epsilon a'_{i}$ is equal to
\[g_{1}^{(i+1)}a'_{1}+h_{1}^{(i+1)}b'_{1}=XA'_{i+1}-\epsilon A'_{i}-w_{1}.\]
Since $XA_{i+1}'=X^{2}A_{i+2}'=\cdots =X^{n-i}A_{n}',$ the coefficient of $x_{1}$ in $XA'_{i+1}$ is in $\epsilon \O$. Thus,
$$
(\lambda'_1g^{(i+1)}_1+\lambda''_1h^{(i+1)}_1+1)w_1+g^{(i+1)}_1x_1+h^{(i+1)}_1y_1\equiv 0\mod \epsilon F_i'.
$$
We must have $g_{1}^{(i+1)}, h_{1}^{(i+1)}\in \epsilon \O$, but then $w_1\equiv  0\mod \epsilon F_i'$, which is impossible. 
Hence, $F'_i$ is indecomposable if $2\leq n-i=i$. 
\end{proof}

\section{Main result}

In this section, we prove the main result of this article. 

\begin{theorem}
Let $\O$ be a complete discrete valuation ring, $A=\O[X]/(X^n)$, for $n\geq 2$. Then, 
the component of the stable Auslander-Reiten quiver of $A$ which contains $Z_i$ and $Z_{n-i}$ is $\Z A_\infty/\langle \tau^2\rangle$ if 
$2i\ne n$, and $\Z A_\infty/\langle \tau\rangle$ i.e. homogeneous tube if $2i=n$.
\end{theorem}
\begin{proof}
Let $C$ be a component of the stable Auslander-Reiten quiver of $A$ that contains a Heller lattice and 
we apply Theorem \ref{infinite tree class theorem} to $C$. If there exists a loop in the component, 
Proposition \ref{E-indec} implies that the endpoint must be a Heller lattice. However, Heller lattices have no loops. 
Thus, $C$ is a valued stable translation quiver and its tree class is one of $A_\infty, B_\infty, C_\infty, D_\infty$ and $A^\infty_\infty$.
If $i=1$ or $i=n-1$, then Proposition \ref {E-indec}(1) implies that the subadditive function considered in the proof of Lemma\;\ref{no loop lemma} 
is not additive. Thus, the tree class of $C$ is $A_{\infty}$. We now assume that $i\neq 1,n-1$. Proposition \ref{E-indec}(2) implies that 
the Heller lattices $Z_{i}$ and $Z_{n-i}$ are on the boundary of the stable Auslander-Reiten quiver, and the tree class can not be $A_{\infty}^{\infty}$.  
If the tree class was one of $B_{\infty}$, $C_{\infty}$ and $D_{\infty}$, then $F_{i}$ or $F_{n-i}$ would have at least
three indecomposable direct summands. But it contradicts Lemma\;\ref{F-indec}. Therefore, the tree class is $A_{\infty}$. Then, 
the component $C$ must be a tube, and the rank is the period of the Heller lattices $Z_i$ and $Z_{n-i}$, which is two if $n-i\neq i$, 
one if $n-i=i$.
\end{proof}

%%%%%%%%%%%%%%%%%%%%%
\section{Appendix} %%
%%%%%%%%%%%%%%%%%%%%%

In this appendix, we prove Proposition \ref{construction of AR-seq}. The proof uses arguments from \cite{Bu} and \cite{R1}. 
As it is clear that (1) implies (2), we show that (2) implies (1). 
Let us consider the injective resolution of $\O$ as an $\O$-module:
$$
0  \longrightarrow \O \xrightarrow{\ \iota\ } \K  \xrightarrow{\ d\ } \K/\O  \longrightarrow 0.
$$
Since $\Ext_{\O} ^1(X,\O)=0$ for any free $\O$-modules of finite rank $X$, we have
\[  0 \longrightarrow \Hom _{\O}(X,\O) \longrightarrow \Hom_{\O}(X,\K)  \longrightarrow \Hom _{\O}(X,\K/\O)  \longrightarrow 0. \]
In particular, if we define functors $D'=\Hom _{\O}(-,\K)$ and $D''=\Hom _{\O}(-,\K/\O)$, then we have the short exact sequence
 \[ 0\longrightarrow D(\Hom _A(M, -))\longrightarrow D'(\Hom _A(M, -)) \longrightarrow D''(\Hom _A(M, -)) \longrightarrow 0, \]
 for any $A$-lattice $M$. We define functors 
$$ \nu'=D'\Hom _A(-,A), \quad \nu''=D''\Hom_A(-,A), $$
which we also call Nakayama functors. 
Applying the Nakayama functors $\nu,\nu',\nu''$ to $M$, we obtain the following exact sequences
 \[ 0 \longrightarrow \nu(M) \longrightarrow \nu'(M) \longrightarrow \nu''(M) \longrightarrow 0,\]
and $0 \longrightarrow \Hom_A(-,\nu(M)) \longrightarrow \Hom_A(-,\nu'(M)) \longrightarrow \Hom_A(-,\nu''(M))$. 
Let $\lambda$ be the functorial isomorphism defined by
 \begin{align*} D(\Hom _A(M,A)\otimes _A -)& =\Hom _{\O}(\Hom _A(M,A)\otimes _A - ,\O)\\
  & \simeq \Hom _A(-, \Hom _{\O}(\Hom_A(M,A),\O)) \\
  & = \Hom _{A}(-,\nu (M)). \end{align*}
We define $\lambda'$ and $\lambda''$ in the similar manner by replacing $\nu$ with $\nu'$ and $\nu''$. Let
 \[ \mu _{M} : \Hom _A(M,A)\otimes _{A}-\longrightarrow \Hom _A(M,-) \] 
be the natural transformation defined by $\phi \otimes x \mapsto (m\mapsto \phi(m)x)$. Then, it induces the following three morphisms of functors
\begin{align*} 
D\mu _{M}\; :&\;\; D\Hom _A(M, -) \longrightarrow D(\Hom _A(M,A)\otimes _A -), \\
D'\mu _{M}\; :&\;\; D'\Hom _A(M, -) \longrightarrow D'(\Hom _A(M,A)\otimes _A -), \\
D''\mu _{M}\; :&\;\; D''\Hom _A(M, -)\longrightarrow D''(\Hom _A(M,A)\otimes _A -). 
\end{align*}
Then, we have the following commutative diagram of functors on $A$-lattices. 

\medskip
\hspace{5mm}
\begin{xy}
(-5,15) *{0}="01",(20,15)*{D\Hom _A(M,-)}="L",(60,15)*{D'\Hom _A(M,-)}="E", (100,15)*{D''\Hom _A(M,-)}="M",(125,15)*{0}="02",
(-5,0)*{0}="03",(20,0)*{\Hom _A(-,\nu(M))}="L2",(60,0)*{\Hom _A(-,\nu '(M))}="nP", (100,0)*{\Hom _A(-,\nu ''(M))}="nM",
,(5,0)*{}="d0",(35,15)*{}="d2",(45,15)*{}="d1",(35,0)*{}="d3",(45,0)*{}="d4",(74,15)*{}="d5",(86,15)*{}="d6",(85,0)*{}="d7",
\ar "01";"L"
\ar "d2";"d1"
\ar "d5";"d6"^{d_*}
\ar "M";"02"
\ar "03";"d0"
\ar "d3";"d4"^{\iota_*}
\ar "nP";"d7"
\ar "L";"L2"^{\lambda \circ D\mu _{M}}
\ar "E";"nP"^{\lambda '\circ D'\mu _{M} }
\ar "M";"nM"^{\lambda ''\circ D''\mu _{M}}
\end{xy}

\bigskip
\noindent
with exact rows, where $\iota_*$ and $d_*$ are given by compositions of $\iota$ and $d$ on the left.

\begin{lemma} Let $X$ be an $A$-lattice. If $M\otimes\K$ is a projective $A\otimes\K$-module, then 
\begin{itemize}
\item[(i)] $D'\mu _{M}(X)$ is an isomorphism and natural in $X$,
\item[(ii)] $D\mu _{M}(X)$ is a monomorphism and natural in $X$,
\item[(iii)] if $M$ is a projective $A$-module, then $D\mu _{M}(X)$ is an isomorphism,
\item[(iv)] $D''\mu _{M}(X)$ is an epimorphism and natural in $X$.
\end{itemize}
Moreover, the sequence 
\[
D\Hom _A(M, X) \xrightarrow{\ \lambda\circ D\mu _{M}(X) \ } \Hom_A(X ,\nu(M))
 \xrightarrow{d_*\circ(\lambda '\circ D'\mu _{M}(X))^{-1}\circ\iota_*} D''\Hom _A(M ,X)
\]
is exact.
\end{lemma}
\begin{proof}
Observe that we have an isomorphism
\[ \Hom _{A\otimes\K}(M\otimes\K, A\otimes\K)\otimes_{A} X \simeq \Hom _{A\otimes\K}(M\otimes\K , X\otimes\K), \]
since $M\otimes\K$ is a projective $A\otimes\K$-module. Thus, 
$\Coker(\mu _{M}(X))$ is a torsion $\O$-module and $D'\Coker(\mu _{M}(X))=0$. Then,
$$ 0 \to D'\Hom _A(M,X) \xrightarrow{D'\mu _{M}(X)} D'(\Hom_ A(M,A)\otimes _{A} X) \to \Ext _A ^1(\Coker(\mu _{M}(X)),\K)=0, $$
proving (i). As $\Coker(\mu _{M}(X))$ is a torsion $\O$-module, (ii) also follows. 
The proof of (iii) is the same as (i). The proof of (iv) is similar. By chasing the diagram above, (i) implies the 
exact sequence. 
\end{proof}

\begin{lemma}
\label{exact seq}
Let $M$ be an $A$-lattice, $p:P\to M$ the projective cover, and we define 
$$ L=D(\Coker(\Hom _A(p,A))). $$
Then, we have the following exact sequence of functors.
$$ 0 \longrightarrow D\Hom_A(M,-)\xrightarrow{\lambda\circ D\mu _{M}(-)}\Hom _A(-,\nu(M)) \longrightarrow \Ext _A^1(-,L)\longrightarrow 0. $$
\end{lemma}
\begin{proof}
We recall the short exact sequence
\[ 0\to L\longrightarrow \nu (P) \longrightarrow \nu (M) \longrightarrow 0. \]
Applying the functor $\Hom _A(X,-)$, for an $A$-lattice $X$, we obtain
\[ \Hom _A(X, \nu ( P ))  \longrightarrow \Hom _A(X,\nu(M)) \longrightarrow \Ext _A^1(X,L) \longrightarrow \Ext _A^1(X,\nu ( P ))=0, \]
since $\nu( P )$ is an injective $A$-lattice. Thus, we have the following diagram with exact rows:

\bigskip
\hspace{10mm}
\begin{xy}
(20,15)*[o]+{\Hom _A(X,\nu ( P ))}="L",(60,15)*[o]+{\Hom _A(X,\nu(M))}="E", (100,15)*[o]+{\Ext _A^1(X,L)}="M",(120,15)*[o]+{0}="02",
(-5,0)*[o]+{0}="03",(20,0)*[o]+{D\Hom _A(M,X)}="L2",(60,0)*[o]+{\Hom _A(X,\nu(M))}="nP", (100,0)*[o]+{D''\Hom _A(M,X)}="nM",
,(5,0)*[o]+{}="d0",(35,15)*[o]+{}="d2",(45,15)*[o]+{}="d1",(35,0)*[o]+{}="d3",(45,0)*[o]+{}="d4",(75,15)*[o]+{}="d5",(89,15)*[o]+{}="d6",(85,0)*[o]+{}="d7",(111,15)*[o]+{}="d8",
\ar "d2";"d1"^{\nu(p)_*}
\ar "d5";"d6"
\ar "d8";"02"
\ar "03";"d0"
\ar "d3";"d4"
\ar "nP";"d7"
\ar @<0.5mm>@{-}"E";"nP"
\ar @<-0.5mm>@{-}"E";"nP"
\end{xy}

\medskip
\noindent
We show that $\nu(p)_*$ factors through $\lambda\circ D\mu_{M}(X):D\Hom _A(M,X)\to \Hom _A(X,\nu(M))$. Consider the commutative diagram

\medskip
\hspace{20mm}
\begin{xy}
(0,15)*[o]+{\Hom _A(M,A)\otimes _AX}="L",(60,15)*[o]+{\Hom _A(P,A)\otimes _{A}X}="E",
(0,0)*[o]+{\Hom _A(M,X)}="03",(60,0)*[o]+{\Hom _A(P,X).}="L2",(42,15)*[o]+{}="d2",(47,0)*[o]+{}="d1",(20,15)*[o]+{}="d3",(15,0)*[o]+{}="d4"
\ar "d3";"d2"^{p^*\otimes{\rm id}_X}
\ar "d4";"d1"^{p^*}
\ar "L";"03"^{\mu _M(X)}
\ar "E";"L2"^{\mu _P(X)}
\end{xy}

\medskip
\noindent
By dualizing the diagram, we obtain the commutative diagram

\medskip
\hspace{20mm}
\begin{xy}
(0,15)*[o]+{\Hom _A(X,\nu(M))}="L",(60,15)*[o]+{\Hom _A(X,\nu( P ))}="E",
(0,0)*[o]+{D\Hom _A(M,X)}="03",(60,0)*[o]+{D\Hom _A(P,X),}="L2",(42,15)*[o]+{}="d2",(45,0)*[o]+{}="d1",(16,15)*[o]+{}="d3",(16,0)*[o]+{}="d4"
\ar "d2";"d3"_{\nu(p)_*}
\ar "d1";"d4"_{Dp^*}
\ar "03";"L"_{\lambda\circ D(\mu _M(X))}
\ar "L2";"E"_{\lambda\circ D(\mu _P(X))}
\end{xy}

\medskip
\noindent
and $\lambda\circ D(\mu _P(X))$ is an isomorphism. Therefore, $\nu(p)_*$ factors through $D\Hom _A(M,X)$. 
Since $\Coker(p^*)$ is an $\O$-submodule of $ \Hom _A(\Ker( p ),X)$, it is a free $\O$ module of finite rank and
$\Ext _{\O} ^1(\Coker(p^*),\O)=0$. It follows that $Dp^*$ is an epimorphism. 
This implies that $\Im(\nu(p)_*)=\Im(\lambda\circ D(\mu _M(X)))$, and we get the desired exact sequence.
\end{proof} 

By Lemma \ref{exact seq}, we have the commutative diagram 
$$\begin{xy}
(-5,15)*[o]+{0}="01",(20,15)*[o]+{\operatorname{Im}((\nu(p)_*)}="L",(60,15)*[o]+{\Hom _A(X,\nu(M))}="E", (100,15)*[o]+{\Ext _A^1(X,L)}="M",(120,15)*[o]+{0}="02",
(-5,0)*[o]+{0}="03",(20,0)*[o]+{\operatorname{Im}(\lambda\circ D(\mu _M(X)))}="L2",(60,0)*[o]+{\Hom _A(X,\nu(M))}="nP", (100,0)*[o]+{D''\Hom _A(M,X),}="nM",
(3,0)*[o]+{}="d0",(35,15)*[o]+{}="d2",(45,15)*[o]+{}="d1",(38,0)*[o]+{}="d3",(45,0)*[o]+{}="d4",
(75,15)*[o]+{}="d5",(89,15)*[o]+{}="d6",(85,0)*[o]+{}="d7",(111,15)*[o]+{}="d8",

\ar "01";"L"
\ar "L";"d1"
\ar "d5";"d6"
\ar "d8";"02"
\ar "03";"d0"
\ar "d3";"d4"
\ar "nP";"d7"
\ar @<0.5mm>@{-}"E";"nP"
\ar @<-0.5mm>@{-}"E";"nP"
\ar @<0.5mm>@{-}"L";"L2"
\ar @<-0.5mm>@{-}"L";"L2"
\end{xy}$$
which implies that there exists a monomorphism $\Ext _A^1(X,L) \to D''\Hom _A(M,X)$. 

We set $X = M$. Then $0\rightarrow \Ext _A^1(M,L) \rightarrow D''\End _A(M)$. 
Since $M$ is indecomposable, $\Soc (D''\End_A(M))$ is a simple $\End _A(M)$-module, and hence there exists an isomorphism
\[ \Soc (\Ext _A ^1(M,L)) \simeq \{ f\in D''(\End _A(M))\ |\ f(\Rad\End _A(M))=0 \}. \] 

We are ready to prove that (2) implies (1) in Proposition \ref{construction of AR-seq}. 
By the condition (2)(i), $0\to L\to E\to M\to 0$ does not split. As $L$ and $M$ are indecomposable by the condition (2)(ii), 
we show that every $f\in \Rad\Hom _A(X,M)$ factors through $E$ under the condition (2)(iii). Consider the commutative diagram

\bigskip
\hspace{20mm}
\begin{xy}
(0,30)*[o]+{0}="05",(20,30)*[o]+{L}="L3",(40,30)*[o]+{F}="F", (60,30)*[o]+{X}="X",(80,30)*[o]+{0}="06",
(0,15)*[o]+{0}="01",(20,15)*[o]+{L}="L",(40,15)*[o]+{E}="E", (60,15)*[o]+{M}="M",(80,15)*[o]+{0}="02",
(0,0)*[o]+{0}="03",(20,0)*[o]+{L}="L2",(40,0)*[o]+{\nu(P)}="nP", (60,0)*[o]+{\nu(M)}="nM",(80,0)*[o]+{0}="04",
\ar "01";"L"
\ar "L";"E"
\ar "E";"M"
\ar "M";"02"
\ar "03";"L2"
\ar "L2";"nP"
\ar "nP";"nM"_{\nu(p)}
\ar "nM";"04"
\ar "05";"L3"
\ar "L3";"F"
\ar "F";"X"
\ar "X";"06"
\ar @{-}@<0.5mm>"L";"L2"
\ar @{-}@<-0.5mm>"L";"L2"
\ar "E";"nP"
\ar "M";"nM"^{\varphi}
\ar @{-}@<0.5mm>"L3";"L"
\ar @{-}@<-0.5mm>"L3";"L"
\ar "F";"E"
\ar "X";"M"^{f}
\end{xy}

\medskip
\noindent
with exact rows, where $F$ is the pullback of $X$ and $E$ over $M$. Let $\xi$ be an element in $\Ext _A ^1(M,L)$ which represents the second sequence. 
Then, the condition (2)(iii) implies that $\Rad\End _A(M)\xi =0$ and $\xi \in \Soc(\Ext _A ^1(M,L))$. Consider the commutative diagram

\medskip
\hspace{20mm}
\begin{xy}
(0,15)*[o]+{0}="01",(20,15)*[o]+{\Ext _A ^1(M,L)}="extML",(60,15)*[o]+{D''\Hom _A(M,M)}="DendM",(9,15)="o",(31,15)="o2",
(0,0)*[o]+{0}="03",(20,0)*[o]+{\Ext _A ^1(X,L)}="extXL",(60,0)*[o]+{D''\Hom _A(M,X).}="DhomMX",(9,0)="o3",(31,0)="o4",(44,0)="o5",
\ar "01";"o"
\ar "o2";"DendM"
\ar "03";"o3"
\ar "o4";"o5"
\ar "extML";"extXL"^{\Ext _A ^1(f,L)}
\ar "DendM";"DhomMX"^{D''\Hom _A(M,f)}
\end{xy}

\medskip
\noindent
Let $\xi'$ be the image of $\xi$ under $\Ext _A ^1(M,L)\to D''\Hom _A(M,M)$. Since 
\[ D''\Hom _A(M,f)(\xi')(\psi)=\xi'(f\psi)\in \xi'(\Rad\End _A(M))=0\]
for $\psi \in \Hom _A (M,X)$, 
we have $\Ext _A^1(f,L)(\xi)=0$. Hence, $0 \to L \to F \to X \to 0$ splits. Then, it implies that $f$ factors through $E$.

\bibliographystyle{amsalpha}

\end{document}